\documentclass[10pt,a4paper]{article}

\usepackage[T1]{fontenc}% for polish hook
\usepackage[a4paper,margin=2.35cm]{geometry}
\usepackage{microtype}

\usepackage{epsf,epsfig,amsfonts,amsgen,amsmath,amstext,amsbsy,amsopn,amsthm,cases,listings}
\usepackage{ebezier,eepic}
\usepackage{xcolor}
\usepackage{multirow}
\usepackage{epstopdf}
\usepackage{graphicx}
\usepackage{pgf,tikz}
\usepackage{pgfplots}
\usepackage{mathrsfs}
\usepackage[bottom]{footmisc}
\usepackage{enumerate}
\usepackage{enumitem}
\usepackage[titletoc]{appendix}
\usepackage{booktabs}
\usepackage{url}
\usepackage{mathtools}
\usepackage{authblk}
\usepackage{amssymb}
\usepackage{empheq}
\usepackage{dsfont}
\usepackage{longtable}
\usepackage{float}

\usepackage{subfigure}
\pgfplotsset{compat=1.18}
\usetikzlibrary{arrows}
\usepackage{aligned-overset}
\usepackage{bm}%for vector
\usepackage{bbm}%for vector
\usepackage{titlesec}
\usepackage[
    backref=page,
    hyperfootnotes=false,
    colorlinks=true,
    linkcolor=blue!45!black,
    citecolor=blue!45!black,
    urlcolor=blue!45!black
]{hyperref}
\usepackage{aliascnt}
\usepackage[nameinlink,capitalise,noabbrev]{cleveref}
\usepackage{algorithm}
\usepackage{algpseudocode}
\usepackage{multicol}

\usepackage[thinc]{esdiff}

\allowdisplaybreaks[1]
\setlist{itemsep=0.15em,topsep=0.25em,parsep=0pt}
\definecolor{uuuuuu}{rgb}{0.27,0.27,0.27}
\definecolor{sqsqsq}{rgb}{0.1255,0.1255,0.1255}

\theoremstyle{plain}

\newaliascnt{theorem}{definition}
\newtheorem{theorem}[theorem]{Theorem}
\aliascntresetthe{theorem}
\newaliascnt{lemma}{definition}
\newtheorem{lemma}[lemma]{Lemma}
\aliascntresetthe{lemma}
\newaliascnt{proposition}{definition}
\newtheorem{proposition}[proposition]{Proposition}
\aliascntresetthe{proposition}
\newaliascnt{corollary}{definition}

\aliascntresetthe{corollary}
\newaliascnt{conjecture}{definition}

\aliascntresetthe{conjecture}
\newaliascnt{claim}{definition}
\newtheorem{claim}[claim]{Claim}
\aliascntresetthe{claim}
\theoremstyle{definition}
\newaliascnt{problem}{definition}

\aliascntresetthe{problem}
\newaliascnt{remark}{definition}

\aliascntresetthe{remark}
\newaliascnt{observation}{definition}

\aliascntresetthe{observation}
\newaliascnt{fact}{definition}
\newtheorem{fact}[fact]{Fact}
\aliascntresetthe{fact}
\newaliascnt{example}{definition}

\aliascntresetthe{example}
\crefname{definition}{Definition}{Definitions}
\crefname{theorem}{Theorem}{Theorems}
\crefname{lemma}{Lemma}{Lemmas}
\crefname{proposition}{Proposition}{Propositions}
\crefname{corollary}{Corollary}{Corollaries}
\crefname{conjecture}{Conjecture}{Conjectures}
\crefname{claim}{Claim}{Claims}
\crefname{problem}{Problem}{Problems}
\crefname{remark}{Remark}{Remarks}
\crefname{observation}{Observation}{Observations}
\crefname{fact}{Fact}{Facts}
\crefname{example}{Example}{Examples}
\crefname{enumi}{part}{parts}

\titleformat{\section}{\normalfont\large\bfseries}{\thesection}{0.75em}{}
\titleformat{\subsection}{\normalfont\normalsize\bfseries}{\thesubsection}{0.75em}{}
\titleformat{\subsubsection}{\normalfont\normalsize\itshape}{\thesubsubsection}{0.75em}{}
\titlespacing*{\section}{0pt}{0.95\baselineskip}{0.45\baselineskip}
\titlespacing*{\subsection}{0pt}{0.75\baselineskip}{0.3\baselineskip}
\titlespacing*{\subsubsection}{0pt}{0.6\baselineskip}{0.25\baselineskip}

\makeatletter
\let\oldthebibliography\thebibliography
\renewcommand{\thebibliography}[1]{%
  \oldthebibliography{#1}%
  \setlength{\itemsep}{0.15em}%
  \setlength{\parskip}{0pt}%
}
\newcommand{\plainfootnotetext}[1]{%
  \begingroup
  \renewcommand\thefootnote{}%
  \long\def\@makefntext##1{\noindent ##1}%
  \footnotetext{#1}%
  \endgroup
}
\makeatother
\newcommand{\ex}{\mathrm{ex}}

\newcommand{\ind}{\mathrm{ind}}

\usepackage{tikz}
\usetikzlibrary{calc}

\tikzset{
  linkedge/.style={
    draw=black!65,
    fill=black!18,
    line width=0.45pt,
    fill opacity=0.55
  },
  maj/.style={
    circle, draw=black, fill=white,
    minimum size=17pt, inner sep=0pt, font=\small
  },
  min/.style={
    circle, draw=black, fill=black, text=white,
    minimum size=17pt, inner sep=0pt, font=\small
  }
}

\newcommand{\FivePos}{%
  \coordinate (1) at (90:1.45);
  \coordinate (2) at (162:1.45);
  \coordinate (3) at (234:1.45);
  \coordinate (4) at (306:1.45);
  \coordinate (5) at (18:1.45);
}

\newcommand{\DrawForbiddenConfig}[3]{%
\begin{tikzpicture}[scale=0.95,baseline=(current bounding box.center)]
  \FivePos
  #1
  \foreach \i in {1,...,5}{
    \def\thisstyle{maj}
    \foreach \j in {#2}{
      \ifnum\i=\j\relax
        \xdef\thisstyle{min}
      \fi
    }
    \node[\thisstyle] at (\i) {$\i$};
  }
  \node[font=\scriptsize] at (0,-1.95) {#3};
\end{tikzpicture}%
}
\begin{document}
\title{\bf\Large Tur\'{a}n numbers of $4$-uniform tight even cycles minus one edge}
\date{\today}
\author[1]{Wanfang~Chen}
\author[2]{Jianfeng~Hou}
\author[1]{Xizhi~Liu}
\author[2]{Yixiao~Zhang}
\author[2]{Hongbin~Zhao}
\affil[1]{\small School of Mathematical Sciences, University of Science and Technology of China, Hefei, China}
\affil[2]{\small Center for Discrete Mathematics, Fuzhou University, Fuzhou, China}

\maketitle 
%%%%%%%%%%%%%%%%%%%%%%%%%%%%%%%%%%%%%%%%%%%%%
\begin{abstract}
    For every integer $k \ge 1$ and sufficiently large $n$, we show that the extremal construction for the Tur\'{a}n number of the $4$-uniform tight cycle of length $4k+2$ minus one edge is a complete odd-bipartite $4$-graph. 
    In particular, since $C_{6}^{4-}$ contains the $4$-uniform expanded triangle as a subgraph, our result extends that of Frankl and Keevash--Sudakov on the Tur\'an density and the Tur\'{a}n number of the $4$-uniform expanded triangle.
    We also show that the Tur\'{a}n density of $C_{4k+2}^{4}$ is $1/2$ for all integers $k \ge 2$, and establish the corresponding stability result. This strengthens the result of Sankar on the Tur\'{a}n density of $C_{4k+2}^{4}$ which holds only for all sufficiently large $k$.
\end{abstract}

%%%%%%%%%%%%%%%%%%%%%%%%%%%%%%%%%%%%%%%%%%%%%
\plainfootnotetext{
%   \textit{Keywords:} Hypergraph, Tur\'{a}n problem, Tight Cycle \\
% \textit{MSC2020:} 05C35, 05C65, 05D05 \\
\textit{Email:} \texttt{a372959313@gmail.com}, \texttt{jfhou@fzu.edu.cn}, \texttt{liuxizhi@ustc.edu.cn}, \texttt{fzuzyx@gmail.com}, \texttt{hbzhao2024@163.com}}
%%%%%%%%%%%%%%%%%%%%%%%%%%%%%%%%%%%%%%%%%%%%%

\section{Introduction}\label{SEC:Introduction}

Given an integer $r\ge 2$, an \emph{$r$-uniform hypergraph} (henceforth an \emph{$r$-graph}) $\mathcal{H}$ is a collection of $r$-subsets of some set $V$. We call $V$ the \emph{vertex set} of $\mathcal{H}$ and denote it by $V(\mathcal{H})$. When $V$ is understood, we usually identify a hypergraph $\mathcal{H}$ with its set of edges. Thus, $|\mathcal{H}|$ represents the number of edges in $\mathcal{H}$.

Given a family $\mathcal{F}$ of $r$-graphs, we say an $r$-graph $\mathcal{H}$ is \emph{$\mathcal{F}$-free}
if it does not contain any member of $\mathcal{F}$ as a subgraph.
The \emph{Tur\'{a}n number} $\mathrm{ex}(n, \mathcal{F})$ of $\mathcal{F}$ is the maximum number of edges in an $\mathcal{F}$-free $r$-graph on $n$ vertices. 
The \emph{Tur\'{a}n density} of $\mathcal{F}$ is defined as $\pi(\mathcal{F})\coloneqq \lim_{n\to\infty}\mathrm{ex}(n,\mathcal{F})/\binom{n}{r}$. 
The existence of this limit follows from a simple averaging argument of Katona--Nemetz--Simonovits~\cite{KNS64}, which shows that $\mathrm{ex}(n,\mathcal{F})/\binom{n}{r}$ is non-increasing in $n$.

For $r=2$, the value $\pi(\mathcal{F})$ is well understood thanks to the classical work of Erd\H{o}s--Stone~\cite{ES46} (see also~\cite{ES66}), which extends Mantel's theorem~\cite{Man07} and Tur\'{a}n's seminal theorem~\cite{Tur41} on complete graphs.
For $r \ge 3$, the picture is less clear: the notoriously difficult conjecture of Tur\'{a}n on $K_4^3$~\cite{Tur41}, the complete $3$-graph on four vertices, remains wide open.
We refer to~\cite{dC94,Fur91,Sid95,Kee11} for surveys and general background. 
Significant effort has been devoted to the case $r=3$, where a number of important exact and asymptotic results are known; see, for example,~\cite{Kee05,LZ09,Raz07,Raz10,BT11,LMP24,BLLP24}.
However, as the uniformity $r$ increases, exact density and stability results become comparatively scarce, partly due to the increasing complexity of the extremal configurations.
Among the few general-uniformity results, Mubayi~\cite{Mubayi06} provided the first infinite family of Tur\'{a}n densities for every $r$, while Pikhurko~\cite{Pik24} determined the asymptotic order of $1-\pi(K_{r+1}^r)$, where $K_{r+1}^r$ denotes the complete $r$-graph on $r+1$ vertices.
For further results on Tur\'{a}n problems in higher uniformities, we refer the reader, for example, to~\cite{dC83,FR85,Kee05,LZ09,Mar09,Sid89cont,Sid97,Pik13,MV16,LP25,CY24,Liu25,HLZZZ25}.

Tur\'{a}n problems for tight cycles and their variants have received considerable attention in recent years.   
Given an integer $\ell > r \ge 3$, the \emph{$r$-uniform tight cycle of length $\ell$}, denoted by $C_{\ell}^{r}$, is an $r$-graph with vertex set $\left\{v_1,\ldots,v_{\ell}\right\}$ and edge set $\left\{v_i v_{i+1} \cdots v_{i+r-1}\colon i\in \mathbb{Z}/ \ell \mathbb{Z}\right\}$. 
Let $C_{\ell}^{r-}$ denote the $r$-graph obtained by removing an edge from $C_{\ell}^{r}$. 
When $r \mid \ell$, $C_{\ell}^{r}$ and $C_{\ell}^{r-}$ are $r$-partite, so their Tur\'{a}n densities are all zero by a theorem of Erd\H{o}s~\cite{Erdos67a}. 
For $r=3$, Mubayi and R\"{o}dl conjectured that $\pi(C_{5}^{3})=2\sqrt{3}-3$~\cite{MPS11,MR02}.
Very recently, Kam\v{c}ev, Letzter, and Pokrovskiy determined $\pi(C_{\ell}^{3})=2\sqrt{3}-3$ for all sufficiently large $\ell$ not divisible by $3$~\cite{KLP24}; this was later improved to all $\ell \ge 7$ not divisible by $3$~\cite{BLLP24}. 
For $3$-uniform tight cycles with one edge removed, the density of $C_{\ell}^{3-}$ was determined for every $\ell \ge 5$ not divisible by $3$ independently in~\cite{LMP24} and~\cite{BLLP24}, improving on the large-$\ell$ result of Balogh--Luo~\cite{BL24}.
Sankar~\cite{San26} extended the method of Kam\v{c}ev, Letzter, and Pokrovskiy to all uniformities $r$ in a sophisticated way and, in particular, proved that $\pi(C_{\ell}^{4})=1/2$ for all sufficiently large $\ell$ not divisible by $4$. 

We now make the result of Sankar more precise by introducing the complete odd-bipartite hypergraph.
Given a partition $[n]=V_1\sqcup V_2$, let the \emph{complete odd-bipartite} $2r$-graph $\mathbb{B}_{2r}^{\mathrm{odd}}(V_1,V_2)$ be the $2r$-graph on $[n]$ whose edge set consists of all $2r$-sets that meet $V_1$ in an odd number of vertices.  When only the part sizes matter, we also write $\mathbb{B}_{2r}^{\mathrm{odd}}(|V_1|,|V_2|)$.  Define
\begin{align*}
b_{2r}(n) \coloneqq \max_{|V_1|+|V_2|=n}\left|\mathbb{B}_{2r}^{\mathrm{odd}}(|V_1|,|V_2|)\right|. 
\end{align*}
When the part sizes attain this maximum, we write $\mathbb{B}_{2r}^{\mathrm{odd}}(n)$ for the corresponding extremal complete odd-bipartite $2r$-graph, up to isomorphism.\footnote{There may be more than one maximizing choice of the part sizes.  In that case, this notation denotes any one of the corresponding extremal complete odd-bipartite $2r$-graphs; all such choices have the same number of edges.  For the $4$-uniform case, this possible non-uniqueness is made explicit after \cref{THM:main_C_4k+2minus}.}  A straightforward calculation shows that $b_{2r}(n) \sim \frac{1}{2}\binom{n}{2r}$ as $n \to \infty$. 

Since this work concerns only $4$-graphs, we write $b(n)$ for $b_4(n)$; equivalently, $b(n)\coloneqq|\mathbb{B}_4^{\mathrm{odd}}(n)|$. 
By~\cite[Proposition~3.15]{San26}, the complete odd-bipartite $4$-graph is $C_{\ell}^{4}$-free whenever $4 \nmid \ell$. 
Consequently, $\mathrm{ex}(n,C_{\ell}^4) \ge b(n)$ for every $\ell$ not divisible by $4$. 
Sankar~\cite{San26} proved that this lower-bound construction is asymptotically sharp and established the corresponding stability result for all sufficiently long tight cycles whose lengths are not divisible by four.

%%%%%%%%%%%%%%%%%
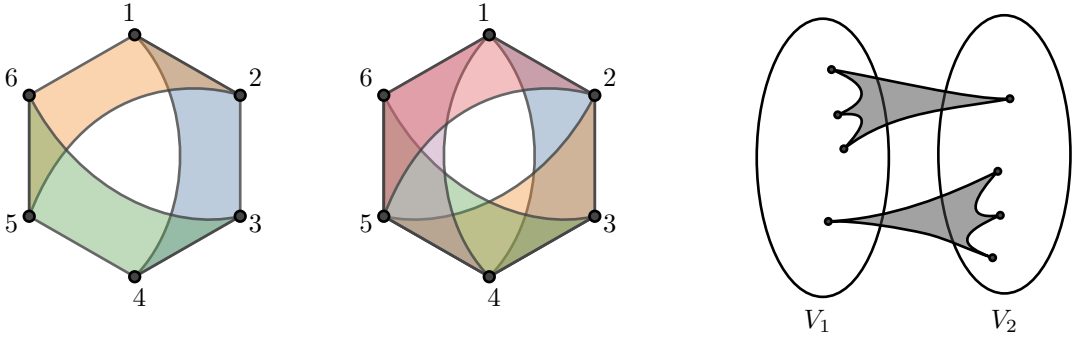
\begin{figure}[H]
\centering
\tikzset{every picture/.style={line width=1pt}} %set default line width to 0.75pt        

\begin{tikzpicture}[x=0.65pt,y=0.65pt,yscale=-1,xscale=1,line join=round]
%uncomment if require: \path (0,300); %set diagram left start at 0, and has height of 300

\definecolor{cycleblue}{RGB}{78,121,167}
\definecolor{cycleorange}{RGB}{242,142,43}
\definecolor{cyclegreen}{RGB}{89,161,79}
\definecolor{cyclepurple}{RGB}{176,122,161}
\definecolor{cyclered}{RGB}{225,87,89}
\tikzset{cycleedge/.style={draw=uuuuuu, draw opacity=0.82, fill opacity=0.38}}

\begin{scope}[shift={(175,0)}]
%Shape: Ellipse [id:dp44524455665673157] 
\draw   (507.85,205.65) .. controls (486.76,205.66) and (469.64,169.64) .. (469.62,125.19) .. controls (469.6,80.75) and (486.68,44.71) .. (507.77,44.7) .. controls (528.86,44.69) and (545.98,80.71) .. (546,125.15) .. controls (546.02,169.6) and (528.94,205.64) .. (507.85,205.65) -- cycle ;
%Shape: Ellipse [id:dp6579161350943095] 
\draw   (402.85,207.65) .. controls (381.76,207.66) and (364.64,171.64) .. (364.62,127.19) .. controls (364.6,82.75) and (381.68,46.71) .. (402.77,46.7) .. controls (423.86,46.69) and (440.98,82.71) .. (441,127.15) .. controls (441.02,171.6) and (423.94,207.64) .. (402.85,207.65) -- cycle ;
%Curve Lines [id:da37694450853202] 
\draw  [fill=uuuuuu, fill opacity=0.5]  (407.85,76.14) .. controls (424.42,82.69) and (435.22,91.97) .. (411.45,102.33) .. controls (432.34,99.6) and (426.32,109.74) .. (415,122) .. controls (443.09,99.63) and (479.31,99) .. (511,93) .. controls (489.4,94.09) and (435.22,80.51) .. (407.85,76.14) --cycle;
%Curve Lines [id:da5288264748438674] 
\draw  [fill=uuuuuu, fill opacity=0.5]  (501,185) .. controls (482.36,178.05) and (479.1,172.04) .. (505.43,160.45) .. controls (482.16,163.67) and (491.6,148.51) .. (504,135)  .. controls (474.51,149.61) and (442.47,162.24) .. (406,164)
 .. controls   (422.94,163.37) and (449,168) .. (462,171)  .. controls (475,174) and (487.72,180.88) .. (501,185) --cycle;
\end{scope}

\begin{scope}[line cap=round,line join=round]
\coordinate (t1) at (180,56);
\coordinate (t2) at (241,91);
\coordinate (t3) at (241,161);
\coordinate (t4) at (180,196);
\coordinate (t5) at (119,161);
\coordinate (t6) at (119,91);
\draw[cycleedge, fill=cycleblue]
  (t1)--(t2)--(t3)--(t4)
  .. controls (215,160) and (215,92) .. (t1)--cycle;
\draw[cycleedge, fill=cycleorange]
  (t1)--(t2)
  .. controls (213,82) and (149,82) .. (t5)
  --(t6)--cycle;
\draw[cycleedge, fill=cyclegreen]
  (t3)--(t4)--(t5)--(t6)
  .. controls (148,148) and (202,170) .. (t3)--cycle;
\foreach \pt in {t1,t2,t3,t4,t5,t6}{\draw[fill=uuuuuu] (\pt) circle (2pt);}
\draw (171,36) node [anchor=north west][inner sep=0.75pt]   [align=left] {$1$};
\draw (244,74) node [anchor=north west][inner sep=0.75pt]   [align=left] {$2$};
\draw (244,158) node [anchor=north west][inner sep=0.75pt]   [align=left] {$3$};
\draw (177,201) node [anchor=north west][inner sep=0.75pt]   [align=left] {$4$};
\draw (103,158) node [anchor=north west][inner sep=0.75pt]   [align=left] {$5$};
\draw (103,74) node [anchor=north west][inner sep=0.75pt]   [align=left] {$6$};
\end{scope}

\begin{scope}[shift={(205,0)},line cap=round,line join=round]
\coordinate (c1) at (180,56);
\coordinate (c2) at (241,91);
\coordinate (c3) at (241,161);
\coordinate (c4) at (180,196);
\coordinate (c5) at (119,161);
\coordinate (c6) at (119,91);
\draw[cycleedge, fill=cycleblue]
  (c1)--(c2)--(c3)--(c4)
  .. controls (215,160) and (215,92) .. (c1)--cycle;
\draw[cycleedge, fill=cycleorange]
  (c2)--(c3)--(c4)--(c5)
  .. controls (158,170) and (212,148) .. (c2)--cycle;
\draw[cycleedge, fill=cyclegreen]
  (c3)--(c4)--(c5)--(c6)
  .. controls (148,148) and (202,170) .. (c3)--cycle;
\draw[cycleedge, fill=cyclepurple]
  (c1) .. controls (145,92) and (145,160) .. (c4)
  --(c5)--(c6)--cycle;
\draw[cycleedge, fill=cyclered]
  (c1)--(c2)
  .. controls (213,82) and (149,82) .. (c5)
  --(c6)--cycle;
\foreach \pt in {c1,c2,c3,c4,c5,c6}{\draw[fill=uuuuuu] (\pt) circle (2pt);}
\draw (171,36) node [anchor=north west][inner sep=0.75pt]   [align=left] {$1$};
\draw (244,74) node [anchor=north west][inner sep=0.75pt]   [align=left] {$2$};
\draw (244,158) node [anchor=north west][inner sep=0.75pt]   [align=left] {$3$};
\draw (177,201) node [anchor=north west][inner sep=0.75pt]   [align=left] {$4$};
\draw (103,158) node [anchor=north west][inner sep=0.75pt]   [align=left] {$5$};
\draw (103,74) node [anchor=north west][inner sep=0.75pt]   [align=left] {$6$};
\end{scope}
%%%%%%%%%%%%%%%%%%% --  --  
\begin{scope}[shift={(175,0)}]
\draw [fill=uuuuuu]  (407.85,76.14) circle (1.2pt);
\draw [fill=uuuuuu] (411.45,102.33) circle (1.2pt);
\draw [fill=uuuuuu] (415,122) circle (1.2pt);
\draw [fill=uuuuuu]  (511,93) circle (1.2pt);
\draw [fill=uuuuuu] (501,185) circle (1.2pt);
\draw [fill=uuuuuu] (505.43,160.45) circle (1.2pt);
\draw [fill=uuuuuu] (504,135) circle (1.2pt);
\draw [fill=uuuuuu] (406,164) circle (1.2pt);
%%%%%%%%%%%%%%%%%%%
% Text Node
\draw (390,213) node [anchor=north west][inner sep=0.75pt]   [align=left] {$V_1$};
% Text Node
\draw (498,213) node [anchor=north west][inner sep=0.75pt]   [align=left] {$V_2$};
\end{scope}

\end{tikzpicture}

\caption{From left to right: the $4$-uniform expanded triangle $\mathcal{C}_{3}^{4}$, the $4$-uniform tight cycle of length $6$ minus one edge $C_6^{4-}$, and the complete odd-bipartite $4$-graph $\mathbb{B}_{4}^{\mathrm{odd}}(n)$.}
\label{fig:expandedK3}
\end{figure}
%%%%%%%%%%%%%%%%%%%

Our first main theorem gives the density and stability statement for the tight cycles in this congruence class.
The extremal model is again the complete odd-bipartite construction.

\begin{theorem}\label{THM:main_C_4k+2}
Let $k \ge 2$ be an integer.  Then $\pi(C_{4k+2}^{4}) = 1/2$. 
Moreover, for every $\delta >0$, there exist $\varepsilon > 0$ and $N_0$ such that every $C_{4k+2}^{4}$-free $4$-graph $\mathcal{H}$ on $n \ge N_0$ vertices with $|\mathcal{H}| \ge \left(1/2 - \varepsilon\right) \binom{n}{4}$ can be transformed into a complete odd-bipartite $4$-graph by adding and deleting at most $\delta n^4$ edges. 
\end{theorem}

The same construction also avoids $C_{\ell}^{4-}$ whenever $\ell \equiv 2 \pmod{4}$; see \cref{lem:odd-bip-Fk-free}.
Together with Sankar's result~\cite{San26}, this gives $\pi(C_{4k+2}^{4-})=1/2$ for all sufficiently large $k$.
Our second main theorem upgrades this asymptotic statement to an exact Tur\'{a}n theorem for every $k\ge 1$ and all sufficiently large $n$.

\begin{theorem}\label{THM:main_C_4k+2minus}
For every integer $k \ge 1$ and sufficiently large $n$, we have $\mathrm{ex}(n,C_{4k+2}^{4-}) = b(n)$, and every $C_{4k+2}^{4-}$-free $4$-graph with $n$ vertices and $b(n)$ edges is complete odd-bipartite. 
\end{theorem}

 We remark that $\mathbb{B}_{4}^{\mathrm{odd}}(n)$ is not an exact extremal construction for the full cycle $C_{4k+2}^{4}$.
Indeed, by \cref{prop:full-cycle-extra-edges}, for every fixed $k\ge1$ and all sufficiently large $n$, one can add $\Omega(n)$ further edges to an extremal complete odd-bipartite $4$-graph without creating a copy of $C_{4k+2}^{4}$.
Also, $\mathbb{B}_{4}^{\mathrm{odd}}(n)$ is not extremal construction for the other two congruence classes $C_{4k+1}^{4-}$ and $C_{4k+3}^{4-}$, since for $n$ large enough, $\mathbb{B}_{4}^{\mathrm{odd}}(n)$ already contains the corresponding cycle with one edge removed (see \cref{SEC:Remarks} for more details).

The value $b(n)=b_4(n)$ can be written explicitly in terms of the part sizes.
If the two parts have sizes $a$ and $n-a$, then $\left|\mathbb{B}_{4}^{\mathrm{odd}}(a,n-a)\right| =a\binom{n-a}{3}+(n-a)\binom{a}{3}$. 
Equivalently, writing $d \coloneqq |2a-n|$, this quantity is
\[
\frac{(n^2-d^2)(n^2+d^2-6n+8)}{48}
=\frac{(n^2-3n+4)^2-\left(d^2-(3n-4)\right)^2}{48}.
\]
Thus the optimal part sizes are those for which $d^2$ is closest to $3n-4$, subject to $d\equiv n\pmod 2$.  In particular, the extremal complete odd-bipartite $4$-graphs are not generally strictly balanced: the two parts differ by $\sqrt{3n}+O(1)$, not by at most one.
Moreover, the optimizing part sizes need not be unique.  Whenever $3n-5=m^2$, the two admissible choices $d=m-1$ and $d=m+1$ both attain the maximum; this gives infinitely many values of $n$ for which the extremal complete odd-bipartite $4$-graph is not unique up to isomorphism.

For $k\ge2$, the preceding stability theorem also applies to $C_{4k+2}^{4-}$, since every $C_{4k+2}^{4-}$-free $4$-graph is $C_{4k+2}^{4}$-free.
The remaining endpoint $k=1$ requires a separate stability statement, which will be used in the exact proof.

\begin{theorem}\label{THM:stabi_C6minus}
For every $\delta >0$, there exist $\varepsilon > 0$ and $N_0$ such that every $C_{6}^{4-}$-free $4$-graph $\mathcal{H}$ on $n \ge N_0$ vertices with $|\mathcal{H}| \ge \left(1/2 - \varepsilon\right) \binom{n}{4}$ can be transformed into a complete odd-bipartite $4$-graph by adding and deleting at most $\delta n^4$ edges. 
\end{theorem}

The case $k=1$ is closely related to the classical Tur\'{a}n problem for expanded triangles. 
For $r \ge 1$, let $\mathcal{C}^{2r}_{3}$ denote the $2r$-graph with vertex set $\{1,\ldots,3r\}$ and edge set 
\begin{align*}
    \left\{\{1,\ldots, r, r+1, \ldots, 2r\}, \{r+1, \ldots, 2r, 2r+1, \ldots, 3r\}, \{1,\ldots, r, 2r+1, \ldots, 3r\}\right\}. 
\end{align*}
Equivalently, $\mathcal{C}^{2r}_{3}$ is obtained from a triangle by replacing each vertex by an $r$-set and each graph edge by the union of the two corresponding $r$-sets.
The Tur\'{a}n problem for $\mathcal{C}^{2r}_{3}$ was first considered by
Frankl~\cite{Frankl90}, who proved that $\pi(\mathcal{C}^{2r}_{3}) = 1/2$. 
Later, Keevash and Sudakov~\cite{KS05a} proved that $\mathrm{ex}(n,\mathcal{C}^{2r}_{3}) = b_{2r}(n)$ for sufficiently large $n$, and every $\mathcal{C}^{2r}_{3}$-free $2r$-graph with $n$ vertices and $b_{2r}(n)$ edges is complete odd-bipartite. 
They also proved the corresponding stability result. 
 
Observe that the expanded triangle $\mathcal{C}^{4}_{3}$ is a subgraph of $C_6^{4-}$ (see \cref{fig:expandedK3}). 
Therefore, the case $k=1$ of \cref{THM:main_C_4k+2minus} strengthens the exact extremal result of Keevash and Sudakov~\cite{KS05a} from $\mathcal{C}^{4}_{3}$-free $4$-graphs to the larger class of $C_6^{4-}$-free $4$-graphs. 
Together with \cref{THM:stabi_C6minus}, this also gives the corresponding stability strengthening.

\paragraph{Organization of the paper.}
In \cref{SEC:Preliminaries} we collect the notation, standard results, and flag algebra setup used throughout the paper.
In \cref{SEC:stability} we prove the two stability results: the density and stability theorem for $C_{4k+2}^4$, and the stability theorem for $C_6^{4-}$.
In \cref{SEC:C6minus_exact} we prove the exact Tur\'{a}n theorem for $C_{4k+2}^{4-}$.
Finally, \cref{SEC:Remarks} contains concluding remarks.

\section{Preliminaries}\label{SEC:Preliminaries}
%%%%%%%%%%%%%%%%%%%%%%%%%%%%%%%%%%%%%%%%%%%%%%%%

Let $H$ be an $r$-graph and $U\subseteq V(H)$ be a subset.
We use $H[U]$ to denote the subgraph of $H$ induced by $U$.  For a vertex $v\in V(H)$, the \emph{link} of $v$ is the $(r-1)$-graph
\[
L_H(v)\coloneqq\left\{S\in\tbinom{V(H)\setminus\{v\}}{r-1}\colon S\cup\{v\}\in H\right\}.
\]
The \emph{degree} of $v$ in $H$ is $d_H(v)\coloneqq |L_H(v)|$, and the \emph{minimum vertex degree} (or \emph{minimum $1$-degree}) of $H$ is $\delta_1(H)\coloneqq \min_{v\in V(H)} d_H(v)$.

Let $H$ and $G$ be two $r$-graphs.  A \emph{homomorphism} from $H$ to $G$ is a map $\varphi:V(H)\to V(G)$ such that $\varphi(e)\in G$ for every $e\in H$.  For an integer $t\ge 1$, the \emph{$t$-blowup} $H[t]$ is obtained by replacing each vertex $x\in V(H)$ with a set $V_x$ of $t$ clones and replacing each edge $x_1\cdots x_r\in H$ with all $r$-sets $y_1\cdots y_r$ such that $y_i\in V_{x_i}$ for every $i$.  Thus an $r$-graph embeds into a blowup of $H$ precisely when it admits a homomorphism to $H$ with bounded preimages.

The next standard blowup result is the mechanism for passing from the two base configurations, $C_{10}^4$ and $C_6^{4-}$, to the longer cycles considered in the main theorems.
Part \textup{(i)} follows from Erd\H{o}s's hypergraph supersaturation theorem~\cite{Erdos67a}, and part \textup{(ii)} is a standard consequence of the hypergraph removal lemma~\cite{RS09}.

\begin{lemma}\label{LEM:blowup_lemma}
Let $F$ be a fixed $r$-graph and let $t\ge1$ be fixed.
The following statements hold.
\begin{enumerate}[label=\textup{(\roman*)}]
\item $\pi(F[t])=\pi(F)$.
\item For every $\eta>0$, there exists $N_0$ such that every $F[t]$-free $r$-graph on $n\ge N_0$ vertices can be made $F$-free by deleting at most $\eta n^r$ edges.
\end{enumerate}
\end{lemma}

The second standard tool is an induced removal lemma.
It will convert small induced densities of forbidden local configurations into a small edit distance from being induced-copy-free.

\begin{lemma}[{\cite{ARS07,RS09}}]\label{LEM:removal_lemma}
Let $r,C\in\mathbb N$ and $\delta>0$ be fixed.  For every family $\mathcal F$ of $r$-graphs on at most $C$ vertices, there exist $\gamma>0$ and $N_0$ such that the following holds.  If an $r$-graph $\mathcal H$ on $n\ge N_0$ vertices contains at most $\gamma n^{|V(F)|}$ induced copies of $F$ for every $F\in\mathcal F$, then $\mathcal H$ can be made induced-$\mathcal F$-free by adding and deleting at most $\delta n^r$ edges.
\end{lemma}

For two sets $A$ and $B$, their \emph{symmetric difference} is $A\triangle B\coloneqq (A\setminus B)\cup(B\setminus A)$. 
For two $r$-graphs $H_1,H_2$ on the same vertex set, their \emph{edit distance} is $\mathrm{edit}(H_1, H_2) \coloneqq \min_{\phi}|\phi(H_1) \triangle H_2|$, where $\phi$ ranges over all bijections between the vertex sets of $H_1$ and $H_2$.  We say that $H_1$ can be transformed into $H_2$ by at most $\eta n^r$ \emph{edits} if $\mathrm{edit}(H_1, H_2) \le \eta n^r$, where $n$ is the common number of vertices.  For a fixed $r$-graph $F$, let $\ind(F,H)$ denote the number of vertex subsets $U\subseteq V(H)$ with $H[U]\cong F$.

The notation $\mathbb{B}_4^{\mathrm{odd}}(X,Y)$ will be used for arbitrary disjoint vertex sets $X,Y$, not only for partitions of $[n]$.

The following fact is the parity reformulation of complete odd-bipartiteness.

\begin{fact}\label{FACT:odd-label}
A $4$-graph $G$ is complete odd-bipartite if and only if there exists a map $\chi:V(G)\to\{0,1\}$ such that, for every $4$-set $x_1x_2x_3x_4\subseteq V(G)$, we have
\begin{equation}\label{EQ:odd_label}
x_1x_2x_3x_4\in G
\quad\text{if and only if}\quad
\chi(x_1)+\chi(x_2)+\chi(x_3)+\chi(x_4)\equiv 1\pmod 2.
\end{equation}
\end{fact}

The parity formulation gives the lower-bound construction for the exact problem.
We record the short freeness argument now, since it will be used throughout the paper.

\begin{lemma}\label{lem:odd-bip-Fk-free}
For every integer $k\ge 1$, every complete odd-bipartite $4$-graph is
$C_{4k+2}^{4-}$-free. Consequently,
\[
\ex(n,C_{4k+2}^{4-})\ge b(n).
\]
\end{lemma}

\begin{proof}
Let $m\coloneqq4k+2$, and let $G\coloneqq\mathbb{B}_4^{\mathrm{odd}}(V_1,V_2)$. Suppose for a contradiction that
$G$ contains a copy of $C_m^{4-}$ on cyclically ordered vertices $x_1,\dots,x_m$.
For each $i\in\mathbb Z/m\mathbb Z$, define $\eta_i\coloneqq1$ if $x_i\in V_1$, and
$\eta_i\coloneqq0$ otherwise. Also set
\[
s_i\coloneqq\eta_i+\eta_{i+1}+\eta_{i+2}+\eta_{i+3}\pmod 2.
\]
The four consecutive vertices $x_i,x_{i+1},x_{i+2},x_{i+3}$ form an edge of $G$ exactly when the corresponding $s_i$ equals $1$.
Since $C_m^{4-}$ has $m-1$ prescribed edges, at least $m-1$ of the numbers $s_i$ are
equal to $1$. But
\[
\sum_{i=1}^{m}s_i\equiv 4\sum_{i=1}^{m}\eta_i\equiv 0\pmod 2,
\]
so the number of $i$ with $s_i=1$ is even. Thus either all $s_i$ are equal to $1$, or
at most $m-2$ of them are equal to $1$.

The first case is impossible. Indeed, if all $s_i=1$, then
$s_{i+1}-s_i\equiv \eta_{i+4}-\eta_i\equiv 0\pmod 2$, so $\eta_{i+4}=\eta_i$ for all $i$.
Since $\gcd(4,m)=2$, the sequence $(\eta_i)$ is $2$-periodic. Hence every $s_i$ is the
sum of two copies of the same two numbers, so $s_i\equiv 0\pmod 2$, a contradiction.

Therefore $G$ contains at most $m-2$ cyclic edges on any cyclically ordered $m$-tuple,
and so $G$ is $C_m^{4-}$-free.
\end{proof}

To use the blowup lemma, we need the following concrete embeddings of the longer cycles into blowups of the two base configurations.

\begin{lemma}\label{LEM:cycle-blowup-embeddings}
The following statements hold.
\begin{enumerate}[label=\textup{(\roman*)}]
\item If $k\ge2$ and $m=4k+2$, then $C_m^4\subseteq C_{10}^4[k-1]$.
\item If $k\ge1$ and $m=4k+2$, then $C_m^{4-}\subseteq C_6^{4-}[k]$.
\end{enumerate}
\end{lemma}

\begin{proof}
For the first embedding, write the vertices of $C_{10}^4$ as $0,1,\ldots,9$ cyclically.  Map the cyclic vertex sequence of $C_m^4$ to
\[
0,1,2,3,4,5,6,7,8,9,
\underbrace{0,1,2,3,\,0,1,2,3,\ldots,0,1,2,3}_{k-2\text{ copies of }0,1,2,3}.
\]
This sequence has length $10+4(k-2)=4k+2$.  Every four consecutive terms, cyclically, form an edge of $C_{10}^4$, and no label occurs more than $k-1$ times.  Hence $C_m^4\subseteq C_{10}^4[k-1]$.

For the second embedding, write the vertices of $C_6^{4-}$ as $1,\ldots,6$ cyclically, with the edge $6123$ deleted.  Map the cyclic vertex sequence of $C_m^{4}$ to
\[
1,2,3,4,5,6,
\underbrace{1,2,3,4,\,1,2,3,4,\ldots,1,2,3,4}_{k-1\text{ copies of }1,2,3,4}.
\]
This sequence has length $6+4(k-1)=4k+2$.  The only four consecutive terms which do not form an edge of $C_6^{4-}$ are $6,1,2,3$, and this exceptional $4$-tuple occurs once.  Deleting the corresponding edge from $C_m^4$ gives a copy of $C_m^{4-}$, and every remaining edge maps to an edge of $C_6^{4-}$.  No label occurs more than $k$ times, so $C_m^{4-}\subseteq C_6^{4-}[k]$.
\end{proof}

Write $[m]\coloneqq\{0,1,\ldots,m-1\}$.  The $4$-graph $\mathcal J_{5,0}$ has vertex set $[5]$, and the $4$-graphs $\mathcal J_{6,i}$ have vertex set $[6]$.  They are defined by their edge sets as follows:
\[
\begin{aligned}
\mathcal J_{5,0}\coloneqq{}&
\{0123,1234,0134,0234,0124\},\\
\mathcal J_{6,0}\coloneqq{}&
\{0135,2345,0345,0123,1234,0134,0124,0145\},\\
\mathcal J_{6,1}\coloneqq{}&
\{1345,0135,1235,0235,2345,0123,1234,0125,1245\},\\
\mathcal J_{6,2}\coloneqq{}&
\{0123,1234,1345,1235,0125,0235,2345\},\\
\mathcal J_{6,3}\coloneqq{}&
\{0345,0123,1234,0125,2345,0145\},\\
\mathcal J_{6,4}\coloneqq{}&
\{0345,0123,1234,0134,1245,2345\},\\
\mathcal J_{6,5}\coloneqq{}&
\{0345,0123,0234,0124,2345\},\\
\mathcal J_{6,6}\coloneqq{}&
\{0245,0345,0123,1234,0125,2345,0145\}.
\end{aligned}
\]
In the verification script these configurations are denoted by $\mathrm{C10\_J5\_0}, \mathrm{C10\_J6\_0},\ldots,\mathrm{C10\_J6\_6}$. 
They are excluded from the $C_{10}^4$ flag algebra computation because each of them admits a homomorphism from $C_{10}^4$.

\begin{lemma}\label{LEM:C10-zero-density-configs}
For each of the configurations $\mathcal J_{5,0}, \mathcal J_{6,0},\ldots,\mathcal J_{6,6}$, there exists a homomorphism from $C_{10}^4$ to it.  Consequently, each of these configurations has limiting density zero in every $C_{10}^4$-free sequence.
\end{lemma}

\begin{proof}
The following table gives explicit homomorphisms from $C_{10}^4$ to these configurations.
A row $(a_0,a_1,\ldots,a_9)$ means that the homomorphism sends $i$ to $a_i$ for each $i\in\{0,1,\ldots,9\}$:
\[
\begin{array}{c|c}
\text{target} & (\varphi(0),\varphi(1),\ldots,\varphi(9))\\ \hline
\mathcal J_{5,0} & (0,1,2,3,0,4,1,2,3,4)\\
\mathcal J_{6,0} & (0,1,2,3,4,1,0,3,5,4)\\
\mathcal J_{6,1} & (0,1,2,3,4,5,1,2,3,5)\\
\mathcal J_{6,2} & (0,1,2,3,4,5,1,3,2,5)\\
\mathcal J_{6,3} & (0,1,2,3,0,1,2,3,4,5)\\
\mathcal J_{6,4} & (0,1,2,3,4,1,2,5,4,3)\\
\mathcal J_{6,5} & (0,1,2,3,0,4,5,3,2,4)\\
\mathcal J_{6,6} & (0,1,2,3,0,1,2,3,4,5).
\end{array}
\]

The displayed rows give the required homomorphisms.  If $J$ is one of the configurations above and $t$ is the maximum preimage size of the displayed homomorphism $C_{10}^4\to J$, then $C_{10}^4\subseteq J[t]$.  Hence every $C_{10}^4$-free $4$-graph is $J[t]$-free.  By \cref{LEM:blowup_lemma}\textup{(ii)}, deleting $o(n^4)$ edges makes such a $4$-graph $J$-free.  Every copy of $J$ in the original $4$-graph must then contain one of the deleted edges, and a fixed deleted edge is contained in $O(n^{|V(J)|-4})$ copies of $J$.  Thus the number of copies of $J$ is $o(n^{|V(J)|})$.
\end{proof}

\bigskip

We use the flag algebra method of Razborov~\cite{Raz07}, also described for hypergraph Tur\'an problems, for example, in~\cite{Raz10,BT11}.  Since the method is now standard, we only recall the form of the computer-assisted certificates used below and refer the reader to the cited papers for the general definitions.  Roughly speaking, for Tur\'an problems in $4$-uniform hypergraphs, a flag algebra proof using $0$-flags on $m$ vertices of an upper bound $u\in\mathbb R$ for a density expression $f$ consists of an asymptotic identity
\begin{equation}\label{EQ:flag_identity}
u-f(H)=\mathrm{SOS}+\sum_{F\in\mathcal F_m^0} c_F\,p(F,H)+o(1),
\end{equation}
valid for every admissible $4$-graph $H$ with $|V(H)|\to\infty$.  Here $\mathrm{SOS}$ is a sum-of-squares term, $\mathcal F_m^0$ is the set of admissible unlabeled $m$-vertex flags, $p(F,H)$ is the induced density of $F$ in $H$, and all coefficients $c_F$ are non-negative.  Thus \eqref{EQ:flag_identity} implies $f(H)\le u+o(1)$ for every admissible sequence.  When $f$ is represented as a linear combination of densities of members of $\mathcal F_m^0$, optimizing the value of $u$ becomes a semidefinite program.

The same identity also gives the local stability information used later.  A flag $F$ has \emph{positive slack} if its coefficient $c_F$ in the exact certificate is positive.  If an admissible sequence attains the bound $u$ asymptotically, then every non-negative term in \eqref{EQ:flag_identity} must vanish in the limit; in particular, every positive-slack flag has limiting density zero.  Therefore the zero-slack flags are the only $m$-vertex induced types that can occur with positive density in an extremal or near-extremal sequence.  This observation is the bridge from the flag algebra certificates below to the local stability statements used in the stability proofs~\cite{PikhurkoSliacanTyros19}.

In each computer-assisted input below, we state the forbidden theory, the target size, the density bound certified by the SDP, the zero-slack flags, and the resulting local stability consequence.  The SDP solutions are rationalized and verified exactly.
All files needed to verify the exact certificates, including the rational positive semidefinite matrices, the admissible type lists, the zero-slack labels, and the verification scripts, are available at
\url{https://github.com/xliu2022/xliu2022.github.io/tree/main/FlagAlgebra_Certificate/FourUniformTightCycles}.
For completeness, the Appendix also gives a human-readable proof of the $C_6^{4-}$ upper bound extracted from the flag algebra certificate.

\section{Proofs of Theorems~\ref{THM:main_C_4k+2} and~\ref{THM:stabi_C6minus}}\label{SEC:stability}

%%%%%%%%%%%%%%%%%%%%%%%%%%%%%%%%%%%%%%%%%%%%%%%%
\subsection{Structure from \texorpdfstring{$6$}{6}-vertex subgraphs}

First, we define the following six $4$-graphs, all with vertex set $\{0,1,2,3,4,5\}$.
\begin{enumerate}[label=\textup{(\roman*)}]
\item $Q_0$ is the $4$-graph with no edge.
\item $Q_1$ has edge set $\{0123,0124,0125,0134,0135,0145,0234,0235,0245,0345\}$. 
\item $Q_2$ has edge set $\{0234,0235,0245,0345,1234,1235,1245,1345\}$. 
\item $Q_3$ has edge set $\{0123,0124,0125,0345,1345,2345\}$. 
\item $Q_4$ has edge set $\{0123,0124,0135,0234,0345,1245,1345\}$. 
\item $Q_5$ has edge set $\{0123,0124,0125,0134,0135,0145,0234,0245,0345\}$. 
\end{enumerate}

\begin{fact}\label{FACT:six-vertex-odd-bipartite-types}
The $4$-graphs $Q_0,Q_1,Q_2,Q_3$ are precisely the induced $6$-vertex subgraphs which occur in complete odd-bipartite $4$-graphs, up to isomorphism.
More explicitly, with vertex set $\{0,1,2,3,4,5\}$, they are realized by the following bipartitions:
\[
\begin{array}{c|c|c}
& A_i & B_i\\
\hline
Q_0 & \varnothing & \{0,1,2,3,4,5\}\\
Q_1 & \{0\} & \{1,2,3,4,5\}\\
Q_2 & \{0,1\} & \{2,3,4,5\}\\
Q_3 & \{0,1,2\} & \{3,4,5\}
\end{array}
\]
in the sense that $Q_i=\mathbb{B}_4^{\mathrm{odd}}(A_i,B_i)$ for the corresponding row.
\end{fact}

The next proposition is the local stability statement extracted from the six-vertex types above.
It says that, once all other six-vertex induced subgraphs are rare, the $4$-graph is close to the complete odd-bipartite template; the exceptional types $Q_4$ and $Q_5$ are removed inside the proof.

\begin{proposition}\label{PROP:structure_from_six_vertex}
For every $\delta>0$, there exist $\varepsilon>0$ and $N_0$ such that the following statement holds.  If $\mathcal H$ is a $4$-graph on $n\ge N_0$ vertices such that $\ind(Q,\mathcal H) \le \varepsilon n^6$ for every $4$-graph $Q$ on $6$ vertices other than the $4$-graphs $Q_0,Q_1,Q_2,Q_3,Q_4,Q_5$, then $\mathcal H$ can be transformed into a complete odd-bipartite $4$-graph by adding and deleting at most $\delta n^4$ edges.
\end{proposition}

The proof of the proposition uses the following local-to-global structural lemma: if only genuine odd-bipartite six-vertex types remain, then the whole $4$-graph is odd-bipartite.

\begin{lemma}\label{LEM:final_structure}
Every $4$-graph $\mathcal G$ on at least $6$ vertices such that every induced $6$-vertex subgraph is isomorphic to one of $Q_0,Q_1,Q_2,Q_3$ is complete odd-bipartite.
\end{lemma}

\begin{proof}
Let $f:\binom{V(\mathcal G)}4\to\{0,1\}$ be the edge indicator of $\mathcal G$, and throughout this proof all sums involving values of $f$ are taken modulo $2$.  For every $S\in\binom{V(\mathcal G)}6$, the induced subgraph $\mathcal G[S]$ is, by hypothesis, isomorphic to one of $Q_0,Q_1,Q_2,Q_3$.  By \cref{FACT:six-vertex-odd-bipartite-types}, these four $4$-graphs are the complete odd-bipartite $4$-graphs on six vertices with part sizes $0|6$, $1|5$, $2|4$, and $3|3$.  Hence, by \cref{FACT:odd-label}, there are labels $\lambda_S(s)\in\{0,1\}$ for $s\in S$ such that, for each $4$-set $U\subseteq S$, we have $f(U)\equiv \sum_{u\in U}\lambda_S(u)\pmod 2$. 
It follows that, for every four distinct vertices $a,b,c,d\in S$, we have
\begin{equation}\label{EQ:six_relation}
f(S\setminus\{a,c\})+f(S\setminus\{a,d\})+f(S\setminus\{b,c\})+f(S\setminus\{b,d\})\equiv 0\pmod 2.
\end{equation}
Indeed, substituting the displayed expression for $f$ on the $6$-set $S$, each label $\lambda_S(s)$ appears an even number of times: twice for $s\in\{a,b,c,d\}$ and four times for $s\in S\setminus\{a,b,c,d\}$. Hence the total sum is congruent to $0$ modulo $2$.

For distinct vertices $x,y\in V(\mathcal G)$, choose a $3$-set $R\subseteq V(\mathcal G)\setminus\{x,y\}$ and define $y_{xy}\coloneqq f(R\cup\{x\})+f(R\cup\{y\})\pmod 2$.
This value is independent of the choice of $R$.  To see this, first suppose that two choices differ in only one vertex, say $R=A\cup\{p\}$ and $R'=A\cup\{q\}$, where $|A|=2$ and $p,q,x,y$ are all outside $A$.  Applying \eqref{EQ:six_relation} to the $6$-set $A\cup\{p,q,x,y\}$ gives $f(Apx)+f(Apy)+f(Aqx)+f(Aqy)\equiv 0\pmod 2$,
which says precisely that the value obtained from $R$ is the same as the value obtained from $R'$.  Any two $3$-sets in $V(\mathcal G)\setminus\{x,y\}$ can be connected by a sequence of such one-vertex changes, so $y_{xy}$ is well defined.

The values $y_{xy}$ satisfy $y_{xy}+y_{yz}+y_{xz}\equiv 0\pmod 2$ for all distinct $x,y,z$.  Indeed, choosing a common $3$-set $R$ disjoint from $\{x,y,z\}$ and expanding the three definitions, each of $f(R\cup\{x\})$, $f(R\cup\{y\})$, and $f(R\cup\{z\})$ appears twice.  Now fix a vertex $v_0$.  Define a global label $\lambda:V(\mathcal G)\to\{0,1\}$ by setting $\lambda(v_0)\coloneqq 0$ and $\lambda(v)\coloneqq y_{v_0v}$ for $v\ne v_0$.
The triangle relation just proved gives $y_{xy}\equiv \lambda(x)+\lambda(y)\pmod 2$ for all distinct $x,y$.

Now define $h(U)\coloneqq f(U)+\sum_{u\in U}\lambda(u)\pmod 2$ for $U\in\binom{V(\mathcal G)}{4}$.
We first show that $h$ is constant on $\binom{V(\mathcal G)}4$.  Suppose that two
$4$-sets $U$ and $W$ share three vertices.  Write $U=R\cup\{x\}$ and $W=R\cup\{y\}$, where $|R|=3$ and $x\ne y$.  By the definition of $y_{xy}$ and by the identity
$y_{xy}\equiv \lambda(x)+\lambda(y)\pmod 2$, we have
\[
\begin{aligned}
h(U)+h(W)
\equiv f(R\cup\{x\})+f(R\cup\{y\})+\lambda(x)+\lambda(y) 
\equiv y_{xy}+y_{xy}
\equiv 0 \pmod 2.
\end{aligned}
\]
Since $h(U),h(W)\in\{0,1\}$, this implies $h(U)=h(W)$.

Now take arbitrary $4$-sets $U,W\subseteq V(\mathcal G)$.  Starting from $U$, one can
replace the vertices of $U\setminus W$ by the vertices of $W\setminus U$, one at a
time, to obtain a sequence of $4$-sets from $U$ to $W$ in which consecutive members
share exactly three vertices.  By the preceding paragraph, $h$ has the same value on
consecutive members of this sequence.  Hence $h(U)=h(W)$.  Thus $h$ is constant; write
$h\equiv c$, where $c\in\{0,1\}$.

It remains to prove that $c=0$.  Fix a $6$-set $S\subseteq V(\mathcal G)$.  For every
$4$-set $U\subseteq S$, the local representation of $f$ on $S$ gives
$f(U)\equiv \sum_{u\in U}\lambda_S(u)\pmod 2$.  Therefore, by the definition of $h$, we have $c\equiv h(U)\equiv \sum_{u\in U}\bigl(\lambda_S(u)+\lambda(u)\bigr)\pmod 2$ for every $U\in\binom S4$.  Set $\tau(s)\coloneqq \lambda_S(s)+\lambda(s)\pmod 2$ for $s\in S$.  Then $\sum_{u\in U}\tau(u)\equiv c\pmod 2$ for every $U\in\binom S4$.

Suppose, for a contradiction, that $c=1$.  We claim first that all values $\tau(s)$,
$s\in S$, are equal.  Indeed, take any two distinct vertices $x,y\in S$.  Since
$|S|=6$, we may choose a $3$-set $R\subseteq S\setminus\{x,y\}$.  Applying the
identity above to the two $4$-sets $R\cup\{x\}$ and $R\cup\{y\}$, we get $\sum_{r\in R}\tau(r)+\tau(x)\equiv 1\pmod 2$ and $\sum_{r\in R}\tau(r)+\tau(y)\equiv 1\pmod 2$.  Subtracting these two congruences gives $\tau(x)\equiv \tau(y)\pmod 2$.
Since $x$ and $y$ were arbitrary, all six labels $\tau(s)$ are equal, say to
$\theta\in\{0,1\}$.  But then every $4$-set $U\subseteq S$ satisfies
$\sum_{u\in U}\tau(u)\equiv 4\theta\equiv 0\pmod 2$, contradicting the assumption that this sum is always congruent to $c=1$.  Hence $c\ne 1$,
and therefore $c=0$.

Consequently, $f(A)\equiv \sum_{a\in A}\lambda(a)\pmod 2$ for every $4$-set $A$.  By \cref{FACT:odd-label}, $\mathcal G$ is complete odd-bipartite.
\end{proof}

\begin{proof}[Proof of \cref{PROP:structure_from_six_vertex}]
Let $\mathcal{F}_{6}^{0}$ denote the set of all $6$-vertex $4$-graphs.
Let
\[
\mathcal F_1\coloneqq\mathcal{F}_{6}^{0}\setminus\{Q_0,Q_1,Q_2,Q_3,Q_4,Q_5\}
\quad\text{and}\quad 
\mathcal F_2\coloneqq\mathcal{F}_{6}^{0}\setminus\{Q_0,Q_1,Q_2,Q_3\}.
\]
Let $T_3$ denote the unique $4$-graph on $5$ vertices with exactly three edges, up to isomorphism.  The uniqueness follows because every $4$-edge on a $5$-vertex set is the complement of one vertex.  We shall use the representative $T_3\coloneqq\{0123,0124,0234\}$.

\begin{claim}\label{CLM:T3_facts}
Among the $4$-graphs $Q_0,Q_1,Q_2,Q_3,Q_4,Q_5$, the $4$-graph $T_3$ occurs as an induced $5$-vertex subgraph only in $Q_4$ and $Q_5$.  Moreover, each of $Q_4$ and $Q_5$ contains an induced copy of $T_3$.
\end{claim}

\begin{proof}
It is enough to count the number of edges in the induced subgraphs obtained by deleting one vertex.  By \cref{FACT:six-vertex-odd-bipartite-types}, the $4$-graphs $Q_0,Q_1,Q_2,Q_3$ are complete odd-bipartite on six vertices, with part sizes $0|6$, $1|5$, $2|4$, and $3|3$, respectively.  After deleting one vertex, the resulting complete odd-bipartite $4$-graph on five vertices has part sizes $0|5$, $1|4$, or $2|3$, up to swapping the parts, and therefore has respectively $0$, $4$, or $2$ edges.  Thus none of $Q_0,Q_1,Q_2,Q_3$ contains an induced $T_3$.

For $Q_4$, deleting the vertex $5$ leaves the induced $4$-graph on $\{0,1,2,3,4\}$ with edge set $\{0123,0124,0234\}$, which is $T_3$.  For $Q_5$, deleting the vertex $1$ leaves the induced $4$-graph on $\{0,2,3,4,5\}$ with edge set $\{0234,0245,0345\}$, which has exactly three edges and hence is isomorphic to $T_3$.
\end{proof}

\begin{claim}\label{CLM:T3_unique}
Let $\mathcal G$ be a $4$-graph such that every induced $6$-vertex subgraph of $\mathcal G$ is isomorphic to one of $Q_0,Q_1,Q_2,Q_3,Q_4,Q_5$.  For every edge $X\in \mathcal G$, there is at most one vertex $u\in V(\mathcal G)\setminus X$ such that $\mathcal G[X\cup\{u\}]\cong T_3$.
\end{claim}

\begin{proof}
Suppose to the contrary that an edge $X=\{v_1,v_2,v_3,v_4\}$ has two distinct extensions $v_5$ and $v_6$ to induced copies of $T_3$.  Let $Q\coloneqq\mathcal G[\{v_1,v_2,v_3,v_4,v_5,v_6\}]$.
By \cref{CLM:T3_facts}, the $4$-graph $Q$ must be isomorphic to $Q_4$ or $Q_5$.

We first rule out $Q_5$.  In $Q$, a vertex among $v_1,v_2,v_3,v_4$ is incident with at most five edges not containing both $v_5$ and $v_6$, since the two induced copies of $T_3$ contribute at most $3+3-1$ such edges.  It is incident with at most three further edges containing both $v_5$ and $v_6$.  Hence its degree in $Q$ is at most $8$.  Each of $v_5$ and $v_6$ is incident with two edges in its corresponding copy of $T_3$, and with at most $\binom42$ further edges containing the pair $v_5v_6$, so these two vertices also have degree at most $8$.  Thus $\Delta(Q)\le 8$.  This is impossible for $Q_5$, because the vertex $0$ of $Q_5$ is contained in all its nine edges.

It remains to exclude $Q_4$.  Since $Q_4$ has seven edges, while the two induced copies $Q-v_6$ and $Q-v_5$ of $T_3$ have three edges each and share exactly the edge $X$, the $4$-graph $Q$ has exactly two edges containing the pair $v_5v_6$.  Each of $v_5$ and $v_6$ has degree two in its own induced $T_3$, and therefore has degree four in $Q$.  Thus, if $Q\cong Q_4$, the two vertices corresponding to $v_5$ and $v_6$ would be two degree-four vertices of $Q_4$ which are contained together in exactly two edges.

However, from $Q_4\coloneqq\{0123,0124,0135,0234,0345,1245,1345\}$, we get the degree sequence
\[
d_{Q_4}(0)=5,
\quad d_{Q_4}(1)=5,
\quad d_{Q_4}(2)=4,
\quad d_{Q_4}(3)=5,
\quad d_{Q_4}(4)=5,
\quad d_{Q_4}(5)=4.
\]
The only degree-four vertices are $2$ and $5$, and they are contained together in only one edge, namely $1245$.  This contradicts the conclusion of the preceding paragraph and proves the claim.
\end{proof}

Apply \cref{LEM:removal_lemma} with parameter $\delta/2$ to the family $\mathcal F_2$, and let $\gamma_2>0$ and $N_2$ be the constants obtained.  Choose $N$ so large that $n\binom n4\le \gamma_2n^6$ for all $n\ge N$.  Next apply \cref{LEM:removal_lemma} with parameter $\delta/2$ to the family $\mathcal F_1$, obtaining constants $\gamma_1>0$ and $N_1$.  Put $\varepsilon\coloneqq\gamma_1$, and take $N_0$ large enough that $N_0\ge \max\{N,N_1,N_2\}$.

Let $\mathcal H$ be a $4$-graph satisfying the assumption.  By the removal application to $\mathcal F_1$, after at most $(\delta/2)n^4$ edits we obtain a $4$-graph $\mathcal H'$ in which every induced $6$-vertex subgraph is isomorphic to one of $Q_0,Q_1,Q_2,Q_3,Q_4,Q_5$.

We first prove the explicit bound $\ind(T_3,\mathcal H')\le \frac13\binom n4$. 
Count pairs $(X,u)$ such that $X\in \mathcal H'$, $u\notin X$, and $\mathcal H'[X\cup\{u\}]\cong T_3$.  By \cref{CLM:T3_unique}, each edge $X$ has at most one such extension, so the number of such pairs is at most $|\mathcal H'|\le\binom n4$.  Every induced copy of $T_3$ has exactly three edges and therefore contributes three such pairs.  This proves the displayed bound.
By \cref{CLM:T3_facts}, every induced copy of $Q_4$ or $Q_5$ contains an induced copy of $T_3$.  Once such a copy of $T_3$ is fixed, there are at most $n$ choices for the sixth vertex, and consequently
\[
\ind(Q_4,\mathcal H')+\ind(Q_5,\mathcal H')\le n \cdot \ind(T_3,\mathcal H')\le n\tbinom n4\le \gamma_2 n^6.
\]
Every member of $\mathcal F_1$ has no induced copy in $\mathcal H'$, while $Q_4$ and $Q_5$ have at most $\gamma_2n^6$ induced copies.  Thus every member of $\mathcal F_2$ has at most $\gamma_2n^6$ induced copies in $\mathcal H'$.  A second application of \cref{LEM:removal_lemma}, now to $\mathcal F_2$, changes at most another $(\delta/2)n^4$ edges and gives a $4$-graph $\mathcal H''$ such that every induced $6$-vertex subgraph is isomorphic to one of $Q_0,Q_1,Q_2,Q_3$.  \cref{LEM:final_structure} implies that $\mathcal H''$ is complete odd-bipartite.  Since
\[
|\mathcal H\triangle \mathcal H''|
\le |\mathcal H\triangle \mathcal H'|+|\mathcal H'\triangle \mathcal H''|
\le \delta n^4,
\]
the proposition follows.
\end{proof}

\subsection[Proof of Theorem]{Proof of \texorpdfstring{\cref{THM:main_C_4k+2}}{Theorem}}\label{SEC:C10}

The exact flag algebra certificate for $C_{10}^4$ has two outputs needed here: the sharp density bound and the list of six-vertex types that can have positive density in a near-extremal sequence.

\begin{lemma}\label{LEM:density_C10}\label{LEM:no_6vertex_C10}
The following statements hold.
\begin{enumerate}[label=\textup{(\roman*)}]
\item We have $\pi(C_{10}^{4}) = 1/2$.
\item For every $\gamma > 0$, there exist $\varepsilon > 0$ and $N_0$ such that the following statement holds.  If $\mathcal H$ is a $C_{10}^{4}$-free $4$-graph on $n \ge N_0$ vertices with $|\mathcal H| \ge \left(1/2-\varepsilon\right)\binom{n}{4}$, then every $4$-graph on $6$ vertices other than the $4$-graphs $Q_0,Q_1,Q_2,Q_3,Q_4,Q_5$ has at most $\gamma n^6$ induced copies in $\mathcal H$.
\end{enumerate}
\end{lemma}

\begin{proof}[Computer-assisted verification]
We run the flag algebra SDP at target size $6$ with objective function $d(e)$, the $4$-edge density, in the theory in which the configurations $\mathcal J_{5,0}, \mathcal J_{6,0},\ldots,\mathcal J_{6,6}$ are forbidden.  By \cref{LEM:C10-zero-density-configs}, these zero-density constraints hold in every $C_{10}^4$-free flag algebra limit, so the resulting upper bound applies to the $C_{10}^4$-free problem.

The exact rational certificate gives $d(e) \le 1/2$.
The complete odd-bipartite construction gives the matching lower bound, so the density assertion follows.

We also use the slack information in the same certificate.  There are $65$ admissible $6$-vertex base flags.  Among them, $59$ have positive slack and exactly six have zero slack.  The zero-slack types are precisely
$Q_0,Q_1,Q_2,Q_3,Q_4,Q_5$.
For comparison with the verification script and certificate, the corresponding script indices of these six types are
\[
\begin{array}{c|cccccc}
\text{type} & Q_0 & Q_1 & Q_2 & Q_3 & Q_4 & Q_5\\
\hline
\text{script index} & 0 & 64 & 62 & 46 & 57 & 63
\end{array}
\]
respectively.
By \cref{FACT:six-vertex-odd-bipartite-types}, the four types $Q_0,Q_1,Q_2,Q_3$ are precisely the complete odd-bipartite $6$-vertex types.
The remaining two zero-slack types, $Q_4$ and $Q_5$, do not embed into the complete odd-bipartite template.

We now explain how this gives the finite local-stability statement in \textup{(ii)}.  First, the configurations excluded in the SDP have limiting density zero in every $C_{10}^4$-free sequence by \cref{LEM:C10-zero-density-configs}; that is, each such configuration $J$ has $o(n^{|V(J)|})$ copies.  Every non-admissible $6$-vertex type contains one of $\mathcal J_{5,0},\mathcal J_{6,0},\ldots,\mathcal J_{6,6}$ as a not necessarily induced subgraph.  Extending a copy of such a forbidden configuration to a $6$-set, if necessary, shows that every induced $6$-vertex type containing it has limiting density zero.  Among the remaining admissible $6$-vertex types, the only zero-slack types are $Q_0,Q_1,Q_2,Q_3,Q_4,Q_5$.  Hence every other $6$-vertex type has limiting induced density zero in any $C_{10}^4$-free sequence whose edge density tends to $1/2$.

Finally, suppose that the $\varepsilon$-$\gamma$ assertion in \textup{(ii)} failed for some $\gamma>0$.  Then there would be a sequence of $C_{10}^4$-free $4$-graphs $\mathcal H_i$ with $|V(\mathcal H_i)|\to\infty$ and edge density tending to $1/2$, together with a fixed $6$-vertex type $Q\notin\{Q_0,\ldots,Q_5\}$, such that $\ind(Q,\mathcal H_i)>\gamma |V(\mathcal H_i)|^6$ along a subsequence.  Passing to a further convergent subsequence in the flag algebra sense, the previous paragraph forces the limiting induced density of $Q$ to be zero, while this lower bound keeps it bounded away from zero.  This contradiction proves \textup{(ii)}.
\end{proof}

The blowup results needed for the reduction from $C_{10}^4$ to longer cycles were recorded in \cref{SEC:Preliminaries}.

\begin{proof}[Proof of \cref{THM:main_C_4k+2}]
We first consider the case $k=2$.  \cref{LEM:density_C10}, together with the complete odd-bipartite lower-bound construction, gives $\pi(C_{10}^4)=1/2$.  It remains in this case to prove stability.  Fix $\delta>0$.  Apply \cref{PROP:structure_from_six_vertex} with parameter $\delta$, and let $\gamma>0$ and $N_1$ be the constants obtained.  Next apply \cref{LEM:no_6vertex_C10} with this value of $\gamma$, obtaining $\varepsilon>0$ and $N_2$.  Let $N_0\coloneqq\max\{N_1,N_2\}$.

Let $\mathcal H$ be a $C_{10}^4$-free $4$-graph satisfying the density assumption, with $n\ge N_0$.  By \cref{LEM:no_6vertex_C10}, every $4$-graph on $6$ vertices other than the $4$-graphs $Q_0,Q_1,Q_2,Q_3,Q_4,Q_5$ has at most $\gamma n^6$ induced copies in $\mathcal H$.  Therefore \cref{PROP:structure_from_six_vertex} transforms $\mathcal H$ into a complete odd-bipartite $4$-graph by adding and deleting at most $\delta n^4$ edges.  This proves the theorem for $k=2$.

Assume now that $k\ge3$ and put $m\coloneqq4k+2$ and $t\coloneqq k-1$.  The complete odd-bipartite construction is $C_m^4$-free because $4\nmid m$, so $\pi(C_m^4)\ge1/2$.

By \cref{LEM:cycle-blowup-embeddings}, $C_m^4\subseteq C_{10}^4[t]$.  \cref{LEM:blowup_lemma}\textup{(i)} and \cref{LEM:density_C10} give $\pi(C_m^4)\le \pi(C_{10}^4[t])=\pi(C_{10}^4)=1/2$, and the density assertion follows.

It remains to prove stability.  Fix $\delta>0$, and let $\varepsilon_{10}>0$ be the constant returned by the $k=2$ stability statement just proved, with parameter $\delta/2$.  Choose $0<\eta<\min\{\delta/2,\varepsilon_{10}/100\}$.
Apply \cref{LEM:blowup_lemma}\textup{(ii)} to $F=C_{10}^4$ and this value of $\eta$, and then choose $\varepsilon>0$ sufficiently small in terms of $\varepsilon_{10}$ and $\eta$.

Let $\mathcal H$ be a $C_m^4$-free $4$-graph on $n$ vertices, where $n$ is sufficiently large, and suppose that $|\mathcal H|\ge \left(\frac12-\varepsilon\right)\binom n4$.
Since $C_m^4\subseteq C_{10}^4[t]$, the $4$-graph $\mathcal H$ is $C_{10}^4[t]$-free.  By \cref{LEM:blowup_lemma}\textup{(ii)}, deleting at most $\eta n^4$ edges gives a $C_{10}^4$-free $4$-graph $\mathcal H_0$.  By the choice of the constants and by taking $n$ large, this $4$-graph still satisfies $|\mathcal H_0|\ge \left(\frac12-\varepsilon_{10}\right)\binom n4$. 
The $k=2$ stability statement just proved transforms $\mathcal H_0$ into a complete odd-bipartite $4$-graph by at most $(\delta/2)n^4$ further edits.  Hence $\mathcal H$ itself is within $\eta n^4+\frac{\delta}{2}n^4\le \delta n^4$ edits of a complete odd-bipartite $4$-graph.  This completes the proof of \cref{THM:main_C_4k+2}.
\end{proof}

\subsection[Proof of Theorem]{Proof of \texorpdfstring{\cref{THM:stabi_C6minus}}{Theorem}}\label{SEC:C6minus}

The corresponding certificate for $C_6^{4-}$ is cleaner: its zero-slack six-vertex types are exactly the four complete odd-bipartite types.

\begin{lemma}\label{LEM:density_C6minus}\label{LEM:no_6vertex_C6minus}
The following statements hold.
\begin{enumerate}[label=\textup{(\roman*)}]
\item We have $\pi(C_{6}^{4-}) = 1/2$.
\item For every $\gamma > 0$, there exist $\varepsilon > 0$ and $N_0$ such that the following statement holds.  If $\mathcal{H}$ is a $C_{6}^{4-}$-free $4$-graph on $n \ge N_0$ vertices with $|\mathcal{H}| \ge \left(\frac{1}{2}-\varepsilon\right)\binom{n}{4}$, then every $4$-graph on $6$ vertices other than the $4$-graphs $Q_0,Q_1,Q_2,Q_3$ has at most $\gamma n^6$ induced copies in $\mathcal{H}$.
\end{enumerate}
\end{lemma}

\begin{proof}[Computer-assisted verification]
We run the flag algebra SDP at target size $6$ for the $C_6^{4-}$-free theory, where $C_6^{4-}$ is the $4$-graph on vertex set $\{1,\ldots,6\}$ with edge set $\{1234,2345,3456,1456,1256\}$.
The exact rational certificate gives $d(e)\le \frac12$.
The complete odd-bipartite construction gives the matching lower bound, proving \textup{(i)}.
For completeness, the Appendix gives a direct proof of the upper bound $d(e)\le \frac12$ by unpacking this part of the certificate into a six-vertex inequality.

The same certificate has $65$ admissible $6$-vertex base flags.  Among them, $61$ have positive slack and exactly four have zero slack.  The zero-slack types are precisely
$Q_0,Q_1,Q_2,Q_3$.
By \cref{FACT:six-vertex-odd-bipartite-types}, these are exactly the $6$-vertex complete odd-bipartite types.
For comparison with the verification script and certificate, the corresponding script indices of these four types are
\[
\begin{array}{c|cccc}
\text{type} & Q_0 & Q_1 & Q_2 & Q_3\\
\hline
\text{script index} & 0 & 64 & 61 & 46
\end{array}
\]
respectively.

Thus every other admissible $6$-vertex flag has positive slack.  Consequently, in every near-extremal $C_6^{4-}$-free sequence, the induced density of every other $6$-vertex type tends to zero.  This is exactly the $\varepsilon$-$\gamma$ formulation in \textup{(ii)}.
\end{proof}

\begin{proof}[Proof of \cref{THM:stabi_C6minus}]
Fix $\delta>0$ and let $\mathcal F\coloneqq\mathcal{F}_{6}^{0}\setminus\{Q_0,Q_1,Q_2,Q_3\}$.  
Apply the induced removal lemma, \cref{LEM:removal_lemma}, with $r=4$, $C=6$, the family $\mathcal F$, and parameter $\delta$.  This gives constants $\gamma>0$ and $N_1$.  Applying \cref{LEM:no_6vertex_C6minus} with this value of $\gamma$ gives constants $\varepsilon>0$ and $N_2$.

Let $\mathcal H$ be a $C_6^{4-}$-free $4$-graph on $n\ge \max\{N_1,N_2,6\}$ vertices such that $|\mathcal H|\ge \left(\frac12-\varepsilon\right)\binom n4$.
By \cref{LEM:no_6vertex_C6minus}, every member of $\mathcal F$ has at most $\gamma n^6$ induced copies in $\mathcal H$.  The induced removal lemma therefore allows us to edit at most $\delta n^4$ edges and obtain a $4$-graph $\mathcal H'$ in which every induced $6$-vertex subgraph is isomorphic to one of $Q_0,Q_1,Q_2,Q_3$.  \cref{LEM:final_structure} implies that $\mathcal H'$ is complete odd-bipartite.  Thus $\mathcal H$ can be transformed into a complete odd-bipartite $4$-graph by at most $\delta n^4$ edits.
\end{proof}

\section{Proof of \texorpdfstring{\cref{THM:main_C_4k+2minus}}{Theorem}}\label{SEC:C6minus_exact}

This section proves the exact Tur\'{a}n result for $C_{4k+2}^{4-}$.  We first record the notation and local colored configurations used in the proof, then prove the key local lemmas, and finally complete the argument through a minimum-degree reduction and a potential-function estimate.

\subsection{Preliminaries}\label{SUBSEC:exact-preliminaries}

Throughout this section, we work in the following two-colored setting.  A \emph{two-colored vertex set} is a vertex set equipped with a fixed bipartition $A\sqcup B$, and a \emph{two-colored hypergraph} is a hypergraph whose vertices are so colored.  If $F=(V_F,E_F,A_F,B_F)$ and $G=(V_G,E_G,A_G,B_G)$ are two-colored hypergraphs, then a \emph{color-preserving homomorphism} from $F$ to $G$ is a map $\varphi:V_F\to V_G$ such that $\varphi(A_F)\subseteq A_G$, $\varphi(B_F)\subseteq B_G$, and $\varphi(e)\in E_G$ for every edge $e\in E_F$.  In particular, such a homomorphism need not be injective.  In the colored setting, a \emph{labeled color-preserving copy} means an injective color-preserving homomorphism.  For uncolored hypergraphs, the same labeled convention is used whenever copies are counted; when only existence is asserted, the labeled and unlabeled distinction is immaterial.

For a two-colored vertex set $A\sqcup B$, define the \emph{complete odd-bipartite $3$-graph} $\mathsf O(A,B)$ and the \emph{complete even-bipartite $3$-graph} $\mathsf E(A,B)$ by
\begin{align*}
\mathsf O(A,B) & \coloneqq\left\{xyz\in\tbinom{A\cup B}{3}\colon |xyz\cap B|\equiv 1\pmod 2\right\}, \\
\mathsf E(A,B) & \coloneqq\left\{xyz\in\tbinom{A\cup B}{3}\colon |xyz\cap B|\equiv 0\pmod 2\right\}.
\end{align*}
Thus $\mathsf O(A,B)$ consists of triples of types $AAB$ and $BBB$, while $\mathsf E(A,B)$ consists of triples of types $AAA$ and $ABB$.

For a colored $3$-graph $G=(V,E,A,B)$, write $G^{\mathrm{sw}}\coloneqq(V,E,B,A)$ for the colored $3$-graph obtained by interchanging the two color classes.
The proof uses the following small colored configurations.
Let $\mathcal P\coloneqq\{P_1,P_2,P_1^{\mathrm{sw}},P_2^{\mathrm{sw}}\}$ and $\mathcal J\coloneqq\{J_1,J_2,J_1^{\mathrm{sw}},J_2^{\mathrm{sw}}\}$, where the colored $3$-graphs are defined as follows:
\[
\begin{array}{c|c|c|c|c}
\text{configuration} & V & E & A & B\\ \hline
P_1 & \{1,2,3,4,5\} & \{123,124,135\} & \{1,2,3,4\} & \{5\}\\
P_2 & \{1,2,3,4,5\} & \{123,124,135\} & \{3,4,5\} & \{1,2\}\\
P_1^{\mathrm{sw}} & \{1,2,3,4,5\} & \{123,124,135\} & \{5\} & \{1,2,3,4\}\\
P_2^{\mathrm{sw}} & \{1,2,3,4,5\} & \{123,124,135\} & \{1,2\} & \{3,4,5\}\\
J_1 & \{1,2,3,4\} & \{123,124\} & \{1,2,3\} & \{4\}\\
J_2 & \{1,2,3,4\} & \{123,134\} & \{1,2\} & \{3,4\}\\
J_1^{\mathrm{sw}} & \{1,2,3,4\} & \{123,124\} & \{4\} & \{1,2,3\}\\
J_2^{\mathrm{sw}} & \{1,2,3,4\} & \{123,134\} & \{3,4\} & \{1,2\}.
\end{array}
\]

The following table gives color-preserving homomorphisms from the $5$-vertex configurations in $\mathcal P$ to the $4$-vertex configurations in $\mathcal J$.  The last column records the images of the vertices $1,2,3,4,5$ under the homomorphism $\varphi$. 
\[
\begin{array}{c|c|c}
\text{source} & \text{target} & (\varphi(1),\varphi(2),\varphi(3),\varphi(4),\varphi(5))\\ \hline
P_1 & J_1 & (1,3,2,2,4)\\
P_2 & J_2 & (3,4,1,1,2)\\
P_1^{\mathrm{sw}} & J_1^{\mathrm{sw}} & (1,3,2,2,4)\\
P_2^{\mathrm{sw}} & J_2^{\mathrm{sw}} & (3,4,1,1,2).
\end{array}
\]

\begin{fact}\label{lem:P-to-J-homomorphisms}
Each member of $\mathcal P$ admits a color-preserving homomorphism to a member of $\mathcal J$.
\end{fact}

\subsection{Tur\'{a}n theorem in colored \texorpdfstring{$3$}{3}-graphs}\label{SUBSEC:exact-key-lemmas}

The main result of this subsection is a stability theorem for dense balanced two-colored $3$-graphs.
It says that if such a two-colored $3$-graph contains few copies of each member of $\mathcal P$, then it is close to either the complete odd-bipartite $3$-graph $\mathsf O(A,B)$ or the complete even-bipartite $3$-graph $\mathsf E(A,B)$.

\begin{proposition}\label{lem:local-link-stability}
For every $\xi>0$, there exist constants $\delta>0$, $\zeta>0$, and $n_{\ref{lem:local-link-stability}}$ such that the following holds. Let $L$ be a two-colored $3$-graph on $n\ge n_{\ref{lem:local-link-stability}}$ vertices with color classes $A\sqcup B$, where $\bigl||A|-|B|\bigr|\le \delta n$. Assume that $|L|\ge \left(\frac12-\delta\right)\binom n3$, and that, for every $F\in\mathcal P$, the number of copies of $F$ in $L$ is at most $\zeta n^5$. Then $\min\Bigl\{|L\triangle \mathsf O(A,B)|,\ |L\triangle \mathsf E(A,B)|\Bigr\} \le \xi n^3$.
\end{proposition}

We first prove the stability conclusion in the special case where the two-colored $3$-graph is already $\mathcal J$-free.

\begin{lemma}\label{lem:J-free-local-link-stability}
For every $\xi_0>0$, there exist $\beta>0$ and $n_{\mathcal J}$ such that the following holds.
Let $L$ be a $\mathcal J$-free two-colored $3$-graph on $n\ge n_{\mathcal J}$ vertices with color classes $A\sqcup B$.
If $\bigl||A|-|B|\bigr|\le \beta n$ and $|L|\ge \left(\frac12-\beta\right)\binom n3$, then $\min\Bigl\{|L\triangle \mathsf O(A,B)|,\ |L\triangle \mathsf E(A,B)|\Bigr\} \le \xi_0 n^3$. 
\end{lemma}

\begin{proof}
Let $V\coloneqq V(L)$.
Let $0<\beta\ll\eta\ll\xi_0$, and suppose that $L$ is a $\mathcal J$-free two-colored $3$-graph with color classes $A\sqcup B$, satisfying $\bigl||A|-|B|\bigr|\le \beta n$ and $|L|\ge \left(\frac12-\beta\right)\binom n3$.
For a pair $e$, write $N_A(e)\coloneqq N_L(e)\cap A$ and $N_B(e)\coloneqq N_L(e)\cap B$. Since $L$ is $\mathcal J$-free, every pair has its link contained in one color class: $N_A(e)=\varnothing$ or $N_B(e)=\varnothing$.
Indeed, for an $AA$-pair this is exactly the exclusion of $J_1$, for a $BB$-pair it is the exclusion of $J_1^{\mathrm{sw}}$, and for a mixed pair it follows from the exclusions of $J_2$ and $J_2^{\mathrm{sw}}$.

Call a pair $e$ \emph{good} if $d_L(e)\ge (1/2-\eta)n$.

\begin{claim}\label{clm:many-good-pairs}
All but $2\beta\eta^{-1}n^2$ pairs are good.
\end{claim}

\begin{proof}
Since every pair has degree at most $\max\{|A|, |B|\} \le (1/2+\beta)n$ (by the previous paragraph), and taking $n_{\mathcal J}\ge\beta^{-1}$, double counting gives
\[
\begin{aligned}
\sum_{e\in\binom V2}\left(\left(\frac{1}{2}+\beta\right)n-d_L(e)\right)
=\left(\frac{1}{2}+\beta\right)n\binom n2-3|L| 
\le \left(\frac{1}{2}+\beta\right)n\binom n2-3\left(\frac{1}{2}-\beta\right)\binom n3 
\le 2\beta n^3.
\end{aligned}
\]
If $M$ is the number of non-good pairs, then each non-good pair $e$ contributes at least $\eta n$ to the last sum, because $d_L(e)<(\frac{1}{2}-\eta)n$.
Thus $M\eta n\le 2\beta n^3$, so $M\le 2\beta\eta^{-1}n^2$.
\end{proof}

For a good pair $e$, let $\tau(e)\in\{A,B\}$ be the unique color class containing its link.
Then $N_L(e)\subseteq \tau(e)$ and, by goodness, $|N_L(e)|=d_L(e)\ge(\frac{1}{2}-\eta)n$.
On the other hand, the balance assumption gives $|\tau(e)|\le(\frac{1}{2}+\beta)n$.
Therefore
\[
|\tau(e)\setminus N_L(e)| \le |\tau(e)|-|N_L(e)| \le (\beta+\eta)n.
\]

\begin{claim}\label{clm:label-propagation}
Let $xy$ be a good same-color pair.  Every good mixed pair sharing $x$ or $y$ has label different from $\tau(xy)$.
\end{claim}

\begin{proof}
By symmetry, suppose $x,y\in A$.  If $\tau(xy)=B$ and some good mixed pair $xb$ has label $B$, then
\[
|N_L(xy)\cap B|+|N_L(xb)\cap B|
\ge 2\left(\frac{1}{2}-\eta\right)n>|B|
\]
for $\eta,\beta$ sufficiently small.  Hence there is $c\in B$ such that $xyc,xbc\in L$, which gives a copy of $J_2$, a contradiction.  Thus every good mixed pair $xb$ has label $A$.  The same argument applies to $yb$.

If $\tau(xy)=A$ and some good mixed pair $xb$ has label $A$, then
\[
|N_L(xy)\cap A|+|N_L(xb)\cap A|
\ge 2\left(\frac{1}{2}-\eta\right)n>|A|,
\]
so there is $a\in A$ such that $xya,xba\in L$, giving a copy of $J_1$.  This is again impossible.  The proof with $A$ and $B$ interchanged is identical.
\end{proof}

\begin{claim}\label{clm:same-color-common-label}
There are subsets $A_0\subseteq A$ and $B_0\subseteq B$ and labels $s_A,s_B\in\{A,B\}$ such that
\[
|A\setminus A_0|+|B\setminus B_0|=O(\beta\eta^{-2}n+\eta n),
\]
every good pair in $\binom{A_0}{2}$ has label $s_A$, and every good pair in $\binom{B_0}{2}$ has label $s_B$.
\end{claim}

\begin{proof}
We prove the assertion for $A$; the proof for $B$ is the same.  Let $A'$ be the set of vertices $x\in A$ incident with at most $\eta n$ non-good pairs inside $A$.  By \cref{clm:many-good-pairs}, we have $|A\setminus A'|\eta n\le 2\cdot 2\beta\eta^{-1}n^2$, and hence $|A\setminus A'|=O(\beta\eta^{-2}n)$.
For every $x\in A'$, the vertex $x$ is incident with a good pair inside $A$.

If $x\in A'$ is incident with good $AA$-pairs of both labels, then \cref{clm:label-propagation} gives incompatible labels for every good mixed pair from $x$ to $B$.  Hence every pair $xb$ with $b\in B$ is non-good.  There are at most $2\beta\eta^{-1}n^2$ non-good pairs altogether, so the number of such vertices is $O(\beta\eta^{-1}n)$, which is $O(\beta\eta^{-2}n)$.  After deleting them from $A'$, every remaining vertex $x$ has a well-defined label $\lambda_A(x)$, namely the common label of the good $AA$-pairs incident with $x$.

Let $A_A$ and $A_B$ be the remaining vertices with $\lambda_A(x)=A$ and $\lambda_A(x)=B$, respectively.  If both sets had size at least $\eta n$, then every pair crossing $A_A$ and $A_B$ would be non-good, since a good crossing pair would have to have both labels.  This would give at least $\eta^2n^2$ non-good pairs, contradicting \cref{clm:many-good-pairs} because $\beta\ll\eta^3$.  Thus one of $A_A,A_B$ has size less than $\eta n$.  Delete this smaller set and call the remaining set $A_0$; its common label is $s_A$.
\end{proof}

\begin{claim}\label{clm:mixed-labels-and-template}
There is a label $s\in\{A,B\}$ such that all but $O(\eta n^2)$ pairs have the following labels: same-color pairs have label $s$, and mixed pairs have the other label.  Consequently, $\min\Bigl\{|L\triangle \mathsf O(A,B)|,\ |L\triangle \mathsf E(A,B)|\Bigr\}=O(\eta n^3)$. 
\end{claim}

\begin{proof}
Let $A_0,B_0,s_A,s_B$ be as in \cref{clm:same-color-common-label}.  All pairs meeting $(A\setminus A_0)\cup(B\setminus B_0)$, together with all non-good pairs, account for only $O(\eta n^2)$ pairs, using $\beta\ll\eta^3$.  Since $|A_0||B_0|\ge n^2/5$ for $n$ large and $\eta$ small, there is at least one good mixed pair between $A_0$ and $B_0$.

Consider a good mixed pair $ab$ with $a\in A_0$ and $b\in B_0$.  Since $a$ is incident with a good $AA$-pair of label $s_A$, \cref{clm:label-propagation} forces $\tau(ab)\ne s_A$.  Similarly, since $b$ is incident with a good $BB$-pair of label $s_B$, it forces $\tau(ab)\ne s_B$.  Applying this to one good mixed pair between $A_0$ and $B_0$ gives $s_A=s_B$; call this common label $s$.  It follows that all good mixed pairs between $A_0$ and $B_0$ have the label different from $s$.

Thus all but $O(\eta n^2)$ pairs have the asserted labels.  If $s=B$, let $\Pi_s=\mathsf O(A,B)$, and if $s=A$, let $\Pi_s=\mathsf E(A,B)$.  Every edge of $L\setminus\Pi_s$ contains a pair which is either non-good or has the wrong label, so $|L\setminus\Pi_s|=O(\eta n^3)$.

Conversely, apart from triples containing a non-good or wrongly labeled pair, a triple of $\Pi_s$ can be missing only because one of its correctly labeled good pairs misses the third vertex.  For each good pair $e$, the number of such missed vertices is at most $(\beta+\eta)n=O(\eta n)$.  Summing over all pairs gives $|\Pi_s\setminus L|=O(\eta n^3)$.
\end{proof}

Choosing $\eta\ll\xi_0$ and then $\beta\ll\eta^3$, and taking $n$ sufficiently large, \cref{clm:mixed-labels-and-template} proves the lemma.
\end{proof}

We now prove \cref{lem:local-link-stability}. 

\begin{proof}[Proof of \cref{lem:local-link-stability}]
Let $V\coloneqq V(L)$.
The goal is to remove a small number of edges from $L$ to obtain a $\mathcal J$-free two-colored $3$-graph, and then apply \cref{lem:J-free-local-link-stability}.
We return to the original $3$-graph $L$. Apply \cref{lem:J-free-local-link-stability} with $\xi_0=\xi/2$, and let $\delta_1>0$ and $n_0$ be the corresponding constants. Choose $0<\theta\ll\min\{\xi,\delta_1\}$.
By the standard colored removal consequence of the hypergraph removal lemma~\cite{RS09}, there exist $\mu>0$ and $n_1$ such that, whenever a two-colored $3$-graph on $n\ge n_1$ vertices contains at most $\mu n^4$ labeled color-preserving copies of each member of $\mathcal J$, deleting at most $\theta n^3$ edges makes it $\mathcal J$-free.
For this value of $\mu$, Erd\H{o}s's hypergraph K\H{o}v\'ari--S\'os--Tur\'an supersaturation theorem~\cite{Erdos67a}, applied in this fixed colored partite setting to the color-preserving homomorphisms listed in \cref{lem:P-to-J-homomorphisms}, gives constants $\zeta_1>0$ and $n_2$ such that if some $J\in\mathcal J$ has more than $\mu n^4$ labeled color-preserving copies in $L$, then the corresponding configuration $F\in\mathcal P$ has more than $\zeta_1n^5$ labeled color-preserving copies in $L$.
Finally choose $\zeta\le\zeta_1$ and $\delta\ll\delta_1$.

\begin{claim}\label{clm:few-J-copies}
Every member of $\mathcal J$ has at most $\mu n^4$ labeled color-preserving copies in $L$.
\end{claim}

\begin{proof}
Suppose that some $J\in\mathcal J$ has more than $\mu n^4$ labeled color-preserving copies in $L$.
By the preceding supersaturation statement, the corresponding configuration $F\in\mathcal P$ has more than $\zeta_1n^5$ labeled color-preserving copies in $L$.
Since $\zeta\le\zeta_1$, this contradicts the assumption of the proposition that every member of $\mathcal P$ has at most $\zeta n^5$ labeled color-preserving copies in $L$.
\end{proof}

Take $n_{\ref{lem:local-link-stability}}\ge\max\{n_0,n_1,n_2\}$.
By \cref{clm:few-J-copies} and the colored removal statement, there is a $\mathcal J$-free two-colored $3$-graph $L'$ on the same vertex set after deleting at most $\theta n^3$ edges from $L$.
Since $\delta$ and $\theta$ were chosen sufficiently small compared with $\delta_1$, we have $|L'|\ge \left(\frac12-\delta_1\right)\binom n3$ for all sufficiently large $n$. Hence $L'$ is within $(\xi/2)n^3$ edges of one of $\mathsf O(A,B)$ and $\mathsf E(A,B)$. Since $L$ and $L'$ differ by at most $\theta n^3$ edges and $\theta\ll\xi$, the $3$-graph $L$ is within $\xi n^3$ edges of the same $3$-graph. This proves the proposition.
\end{proof}

\subsection{Proof of \texorpdfstring{\cref{THM:main_C_4k+2minus}}{Theorem}}\label{SUBSEC:exact-main-proof}

We now complete the proof of \cref{THM:main_C_4k+2minus}.
The colored stability result from the previous subsection will be applied to vertex links of a near-extremal $4$-graph.
Before doing this, we record two auxiliary steps.
The first reduces the problem to the natural minimum-degree setting, and the second shows that a small forbidden configuration in one vertex link extends to many copies of $C_{4k+2}^{4-}$.

The following lemma shows that, in the proof of \cref{THM:main_C_4k+2minus}, we may assume that an extremal $4$-graph has minimum degree at least $b(n)-b(n-1) = (1/2+o(1))\binom n3$.

\begin{lemma}\label{lem:min-degree-reduction}
Fix $k\ge1$. Suppose that there exists $N$
such that the following reduced statement holds for every $n\ge N$: if an
$C_{4k+2}^{4-}$-free $4$-graph $G$ on $n$ vertices satisfies $\delta_1(G)\ge b(n)-b(n-1)$ and $|G|\ge b(n)$, then there is a partition $X\sqcup Y=V(G)$ attaining
$b(n)$ such that $G=\mathbb{B}_4^{\mathrm{odd}}(X,Y)$.
Then, for every sufficiently large $n$, every $C_{4k+2}^{4-}$-free $4$-graph $H$
on $n$ vertices with at least $b(n)$ edges satisfies the following:
there is a partition $X\sqcup Y=V(H)$ attaining $b(n)$ such that $H=\mathbb{B}_4^{\mathrm{odd}}(X,Y)$.
In particular, $\ex(n,C_{4k+2}^{4-})\le b(n)$ for all sufficiently large $n$.
\end{lemma}

\begin{proof}
We first prove the upper bound. Let $n\ge N+N^4$, and suppose that $G_n$ is a
$C_{4k+2}^{4-}$-free $4$-graph on $n$ vertices with $|G_n|>b(n)$. If $G_n$ satisfies the minimum-degree condition, then the reduced statement implies that
$G_n\cong \mathbb{B}_4^{\mathrm{odd}}(n)$ and has exactly $b(n)$ edges, a contradiction.
Otherwise, there exists a vertex $v_n$ such that $d_{G_n}(v_n)<b(n)-b(n-1)$.
Since degrees and $b(\cdot)$ are integers, in fact $d_{G_n}(v_n)\le b(n)-b(n-1)-1$.
Hence
\[
|G_n-v_n|
\ge |G_n|-\bigl(b(n)-b(n-1)-1\bigr)
>b(n-1)+1.
\]
Iterating this deletion procedure, either at some order $t\ge N$ we obtain an
$C_{4k+2}^{4-}$-free $4$-graph $G_t$ with $|G_t|>b(t)$ and
$\delta_1(G_t)\ge b(t)-b(t-1)$, contradicting the reduced statement, or we delete
all the way down to $N$ vertices. In the latter case the excess above $b(\cdot)$
increases by at least $1$ at each deletion, so the final $N$-vertex $4$-graph has
more than $b(N)+N^4$ edges, which is impossible since it has at most $\binom N4<N^4$
edges. Thus every $C_{4k+2}^{4-}$-free $4$-graph on $n$ vertices has at most
$b(n)$ edges for all sufficiently large $n$.

Now let $H$ be a $C_{4k+2}^{4-}$-free $4$-graph on $n$ vertices, where $n$ is sufficiently large, and suppose that $|H|\ge b(n)$. By the upper bound just proved, $|H|=b(n)$. We may also assume that
$n-1$ is large enough for the upper bound to hold. If some vertex $v$ satisfies $d_H(v)<b(n)-b(n-1)$,
then, again using integrality,
\[
 |H-v|\ge b(n)-\bigl(b(n)-b(n-1)-1\bigr)=b(n-1)+1,
\]
contradicting the upper bound at order $n-1$. Hence $H$ satisfies
the minimum-degree condition, and the reduced statement applies. Therefore, there is a partition $X\sqcup Y=V(H)$ attaining $b(n)$ such that $H=\mathbb{B}_4^{\mathrm{odd}}(X,Y)$.
\end{proof}

We next prove the auxiliary lemmas used in the reduced minimum-degree argument.
The first local lemma turns a copy of a member of $\mathcal P$ in one vertex link into many copies of the forbidden hypergraph.

\begin{lemma}\label{lem:P-extension}
Fix $k\ge1$, and put $m\coloneqq4k+2$. There is a constant $c=c(k)>0$ such that the following holds for all sufficiently large $n$. Let $V(H)=A\sqcup B\sqcup\{v\}$, where $|A|,|B|\ge n/3$, and suppose that $H-v=\mathbb{B}_4^{\mathrm{odd}}(A,B)$.
If $L_H(v)$ contains a color-preserving copy of some $F\in\mathcal P$, then $H$ contains at least $c n^{m-6}$ labeled copies of $C_m^{4-}$ containing $v$.  Moreover, in each such copy, the $4$-edges containing $v$ are exactly the three edges obtained by adjoining $v$ to the three triples of this copy of $F$ in $L_H(v)$.
\end{lemma}

\begin{proof}
It is enough to consider $P_1$ and $P_2$, since the configurations obtained by interchanging the roles of $A$ and $B$ are handled in the same way. This interchange does not change the complete odd-bipartite $4$-graph, because the edge size is even.

Suppose first that $F=P_1$. After relabeling the image of this copy in $L_H(v)$, we may write its vertices as $1,2,3,4\in A$ and $5\in B$, and the three link triples are $123,124,135$. Thus $123v$, $124v$, and $135v$ are edges of $H$. To check which prescribed edges lie in the odd-bipartite construction, regard $v$ as having color $B$. In the cyclic order $v,2,4,(B,A,A,A)^{k-1},5,3,1$, where the middle block is empty if $k=1$, delete the edge corresponding to the first four consecutive vertices. The three remaining edges containing $v$ in the cyclic order are $531v, 31v2, 1v24$, which are exactly the three edges through $v$ corresponding to the link triples $135,123,124$. Every other prescribed edge uses only old vertices. These old vertices lie in the parts $A,A,(B,A,A,A)^{k-1},B,A,A$, and every four consecutive terms in this sequence contain an odd number of vertices in $B$. Hence all old edges used in the construction lie in $\mathbb{B}_4^{\mathrm{odd}}(A,B)$.

For $F=P_2$, relabel the image in $L_H(v)$ so that $1,2\in B$ and $3,4,5\in A$, again with link triples $123,124,135$. Thus the same three edges through $v$ are present in $H$. Again regard $v$ as having color $B$, and use the cyclic order $v,2,4,(A,A,B,A)^{k-1},5,3,1$. The three edges through $v$ in the cyclic order are again $531v,31v2,1v24$, and the old vertices lie in the parts $B,A,(A,A,B,A)^{k-1},A,A,B$. Again every four consecutive terms in this sequence contain an odd number of vertices in $B$, so all old edges lie in $\mathbb{B}_4^{\mathrm{odd}}(A,B)$.

In either case, once the five vertices of the copy and $v$ are fixed, there remain exactly $m-6$ old positions in the chosen cyclic order. Each of these positions has a prescribed color. Since $|A|,|B|\ge n/3$, for all sufficiently large $n$ the remaining positions can be filled in at least $c n^{m-6}$ labeled ways, avoiding the already chosen vertices, for some constant $c=c(k)>0$. All prescribed edges using only old vertices are template edges and hence are present in $H$, while the three prescribed edges through $v$ are present because they come from $L_H(v)$. This proves the lemma.
\end{proof}

We now prove the exact theorem.

\begin{proof}[Proof of \cref{THM:main_C_4k+2minus}]
Put $m\coloneqq4k+2$. By \cref{lem:odd-bip-Fk-free}, every complete odd-bipartite $4$-graph is $C_{4k+2}^{4-}$-free, so it remains to prove the matching upper bound and the equality statement. By \cref{lem:min-degree-reduction}, it is enough to verify the reduced minimum-degree statement: for all sufficiently large $n$, every $C_{4k+2}^{4-}$-free $4$-graph $H$ on $n$ vertices satisfying $\delta_1(H)\ge b(n)-b(n-1)$ and $|H|\ge b(n)$ satisfies $H\cong \mathbb{B}_4^{\mathrm{odd}}(n)$.

Let $H$ be such a $4$-graph, and put $V\coloneqq V(H)$. Choose constants as follows. First choose $0<\varepsilon_3\ll \varepsilon_2\ll \varepsilon_1\ll c_0\ll_k 1$. 
Put $\xi\coloneqq\min\{\varepsilon_3/4,10^{-5}\}$,
and let $\delta_0>0$, $\zeta_0>0$, and $n_{\ref{lem:local-link-stability}}$ be the constants returned by \cref{lem:local-link-stability} for this value of $\xi$. Set
$\alpha_{\mathrm L}\coloneqq\min\{\delta_0,10^{-2}\}$.
Finally choose $0<\varepsilon_4\ll_k \min\{\varepsilon_3,\zeta_0,\alpha_{\mathrm L}^4\}$. 
Let $\rho>0$ be the density tolerance supplied by \cref{THM:stabi_C6minus} with edit parameter $\varepsilon_4/2$.  Choose
$0<\eta<\min\{\varepsilon_4/2,\rho/100\}$.
Apply \cref{LEM:blowup_lemma}\textup{(ii)} with $F=C_6^{4-}$, $t=k$, and this value of $\eta$.  Assume $n$ is sufficiently large compared with all fixed constants above and with the thresholds obtained from this application and from \cref{THM:stabi_C6minus}.

By \cref{LEM:cycle-blowup-embeddings}, $C_{4k+2}^{4-}\subseteq C_6^{4-}[k]$.  Since $H$ is $C_{4k+2}^{4-}$-free, it is $C_6^{4-}[k]$-free.  Hence, by \cref{LEM:blowup_lemma}\textup{(ii)}, deleting at most $\eta n^4$ edges gives a $C_6^{4-}$-free $4$-graph $H_0$.  Since $b(n)=\left(\frac12+o(1)\right)\binom n4$ and $\eta<\rho/100$, for all sufficiently large $n$ we have $|H_0|\ge b(n)-\eta n^4\ge \left(\frac12-\rho\right)\binom n4$. 
\cref{THM:stabi_C6minus} gives a partition $U_1\sqcup U_2=V$ such that, with $\mathcal T_U\coloneqq\mathbb{B}_4^{\mathrm{odd}}(U_1,U_2)$, we have $|H\triangle \mathcal T_U|\le |H\triangle H_0|+|H_0\triangle \mathcal T_U|\le \varepsilon_4n^4$. 

Now choose a partition $V_1\sqcup V_2=V$ maximizing $|H\cap \mathbb{B}_4^{\mathrm{odd}}(V_1,V_2)|$. 
Put
\[
\mathcal T\coloneqq\mathbb{B}_4^{\mathrm{odd}}(V_1,V_2),\qquad
\mathcal B\coloneqq H\setminus \mathcal T,\qquad
\mathcal M\coloneqq\mathcal T\setminus H.
\]
Here $\mathcal T$ is the \emph{template}, $\mathcal B$ is the family of \emph{bad edges}, and $\mathcal M$ is the family of \emph{missing template edges}.
We shall repeatedly use the decomposition
\begin{equation}\label{eq:H-decomposition-final}
|H|=|\mathcal T|-|\mathcal M|+|\mathcal B|.
\end{equation}

\begin{claim}\label{clm:BM-upper}
We have $\max\{|\mathcal B|,|\mathcal M|\}\le \varepsilon_4n^4$. 
\end{claim}

\begin{proof}
By the maximality of $V_1\sqcup V_2$, we have $|H\cap \mathcal T|\ge |H\cap \mathcal T_U|$. 
Since $|H\triangle \mathcal T_U|\le \varepsilon_4n^4$, we get
\[
|\mathcal B|=|H|-|H\cap \mathcal T|
\le |H|-|H\cap \mathcal T_U|
\le \varepsilon_4n^4.
\]
Also $|H|\ge b(n)\ge |\mathcal T|$, and therefore $|\mathcal M|=|\mathcal T|-|H\cap \mathcal T| \le |H|-|H\cap \mathcal T|=|\mathcal B|$. 
This proves the claim.
\end{proof}

The next claim shows that the maximizing partition is balanced enough for the local link stability lemma.

\begin{claim}\label{clm:part-size}
We have $\bigl||V_1|-|V_2|\bigr|\le {\alpha_{\mathrm L} n}/{20}$. 
In particular, $\min\{|V_1|,|V_2|\} \ge n/3$.
\end{claim}

\begin{proof}
Put $n_i\coloneqq|V_i|$ for $i=1,2$, and $x\coloneqq n_1/n$. Since $|H|\ge b(n)$ and $|\mathcal B|\le\varepsilon_4n^4$, we have 
\begin{equation}\label{eq:starbound}
|\mathcal T|
\ge |H\cap \mathcal T|
=|H|-|\mathcal B|
\ge b(n)-\varepsilon_4n^4.
\end{equation}
On the other hand,
\[
|\mathcal T|=n_1\binom{n_2}{3}+n_2\binom{n_1}{3}
=\frac{n^4}{6}\bigl(x(1-x)^3+x^3(1-x)\bigr)+O(n^3).
\]
Using the identity $x(1-x)^3+x^3(1-x)=\frac18-2\left(x-\frac12\right)^4$ and $b(n)=n^4/48+O(n^3)$, inequality \eqref{eq:starbound} gives
\[
\frac{n^4}{48}-\frac{n^4}{3}\left(x-\frac12\right)^4+O(n^3)
\ge \frac{n^4}{48}-\varepsilon_4n^4-O(n^3).
\]
Hence $|x-1/2|^4\le C\varepsilon_4+O(1/n)$ for an absolute constant $C$. By the choice of $\varepsilon_4$ and by taking $n$ large, $|x-1/2|\le \alpha_{\mathrm L}/40$. Therefore $\bigl||V_1|-|V_2|\bigr|=2n\left|x-\frac12\right|\le \frac{\alpha_{\mathrm L}}{20}n$.
Since $\alpha_{\mathrm L}\le10^{-2}$, both parts have size at least $0.49n$ for large $n$, and in particular both have size at least $n/3$.
\end{proof}

The next claim gives vertex-wise control of missing and bad edges. It uses both the minimum-degree assumption and the maximality of $V_1\sqcup V_2$.

\begin{claim}\label{clm:vertex-missing}
For every vertex $v\in V$, we have $d_{\mathcal M}(v)+d_{\mathcal B}(v)\le \varepsilon_3n^3$. 
In particular, $d_{\mathcal M}(v)\le \varepsilon_3n^3$.
\end{claim}

\begin{proof}
We use the constants $\xi$, $\delta_0$, $\zeta_0$, and $n_{\ref{lem:local-link-stability}}$ fixed above. By \cref{clm:part-size}, after deleting any fixed vertex the two color classes are $\delta_0$-balanced for all sufficiently large $n$.

Fix $v\in V$. By symmetry between the two parts, we may assume that $v\in V_1$; the case $v\in V_2$ is obtained by interchanging $V_1$ and $V_2$, which does not change the complete odd-bipartite $4$-graph. Set $X\coloneqq V\setminus\{v\}$ and $L\coloneqq L_H(v)$, and regard $L$ as a two-colored $3$-graph on $X$ with the colors inherited from $V_1\sqcup V_2$. Put $A\coloneqq V_1\setminus\{v\}$ and $B\coloneqq V_2$. 
Since $v\in V_1$, the triples $T\in\mathsf O(A,B)$ are precisely those for which $T\cup\{v\}\in\mathcal T$, and $\binom X3\setminus\mathsf O(A,B)=\mathsf E(A,B)$. Moreover,
\begin{equation}\label{eq:V1}
\left|\mathsf O(A,B)\setminus L\right|=d_{\mathcal M}(v),
\quad\text{and}\quad 
\left|L\setminus\mathsf O(A,B)\right|=d_{\mathcal B}(v).
\end{equation}

We verify the hypotheses of \cref{lem:local-link-stability} for $L$. By the minimum-degree assumption, we have 
\[
|L|=d_H(v)\ge b(n)-b(n-1)
=\left(\frac12+O(1/n)\right)\binom{n-1}{3}
\ge \left(\frac12-\delta_0\right)\binom{n-1}{3}.
\]
It remains to bound the number of copies of members of $\mathcal P$ in $L$. Fix $F\in\mathcal P$, and let $N_F(L)$ be the number of color-preserving labeled copies of $F$ in $L$. Define the following auxiliary $4$-graph on $A\cup B\cup\{v\}$: 
\begin{align*}
    G_v\coloneqq\mathbb{B}_4^{\mathrm{odd}}(A,B)\cup\left\{ e \cup \{v\} \colon e \in L\right\}. 
\end{align*}
Since $L_{G_v}(v)=L$, and since \cref{clm:part-size} gives $|A|,|B|\ge n/3$, \cref{lem:P-extension} applied to $G_v$ shows that each copy of $F$ in $L$ gives at least $c_k n^{m-6}$ labeled potential copies of $C_m^{4-}$ in $G_v$, where $c_k>0$ depends only on $k$.

All $4$-edges not containing $v$ that are used in these potential copies are edges of $\mathcal T$ contained in $X$. If none of those edges belonged to $\mathcal M$, then all prescribed edges of such a copy would be present in $H$: the edges not containing $v$ would be present template edges, and the three edges through $v$ are present because their triples lie in $L=L_H(v)$. Since $H$ is $C_m^{4-}$-free, every potential copy counted above must contain at least one missing template edge from $\mathcal M$ that is contained in $X$. Conversely, a fixed edge of $\mathcal M$ can lie in at most $D_k n^{m-5}$ such labeled potential copies, where $D_k$ depends only on $k$. Therefore 
\[
N_F(L)c_k n^{m-6}
\le D_k|\mathcal M|n^{m-5}
\le D_k\varepsilon_4 n^{m-1},
\]
and hence $N_F(L)\le D'_k\varepsilon_4 n^5\le \zeta_0(n-1)^5$, by the choice of $\varepsilon_4$ and $n$. Thus \cref{lem:local-link-stability} applies to $L$, giving
\begin{equation}\label{eq:V4}
\min\left\{|L\triangle\mathsf O(A,B)|,~
\left|L\triangle\left(\tbinom X3\setminus\mathsf O(A,B)\right)\right|\right\}
\le \xi(n-1)^3.
\end{equation}

We claim that the second alternative in \eqref{eq:V4} is impossible. Let $g_v\coloneqq|L\cap\mathsf O(A,B)|$ and $h_v\coloneqq\left|L\cap\left(\tbinom X3\setminus\mathsf O(A,B)\right)\right|$. 
If we move only the vertex $v$ to the other side of the partition, all edges of $H$ not containing $v$ keep the same status with respect to the odd-bipartite template, while the two classes $\mathsf O(A,B)$ and $\binom X3\setminus\mathsf O(A,B)$ are interchanged for triples through $v$. Since $V_1\sqcup V_2$ maximizes $|H\cap \mathcal T|$, we must have $g_v\ge h_v$.
The balance estimate from \cref{clm:part-size} gives $|V_1|,|V_2|\ge0.49n$ for all sufficiently large $n$. Since $v\in V_1$, we have $|A|\ge0.48n$ and $|B|\ge0.49n$. Hence
\[
|\mathsf O(A,B)|\ge \tbinom{|B|}{3}\ge10^{-3}n^3
\quad\text{and}\quad
\left|\tbinom X3\setminus\mathsf O(A,B)\right|=|\mathsf E(A,B)|\ge \tbinom{|A|}{3}\ge10^{-3}n^3
\]
for all sufficiently large $n$. Therefore, if the second alternative in \eqref{eq:V4} held, then $g_v\le \xi(n-1)^3$ and $h_v\ge 10^{-3}n^3-\xi(n-1)^3$, which contradicts $g_v\ge h_v$ because $\xi\le10^{-5}$. Therefore \eqref{eq:V4} gives $|L\triangle\mathsf O(A,B)|\le \xi(n-1)^3$.
Together with \eqref{eq:V1}, this yields
\[
d_{\mathcal M}(v)+d_{\mathcal B}(v)=|L\triangle\mathsf O(A,B)|
\le \xi(n-1)^3
\le \varepsilon_3n^3.
\]
This proves the claim.
\end{proof}

We next show that each bad edge forces many missing template edges nearby.
By \cref{clm:part-size,clm:vertex-missing}, the hypotheses needed below are now available.

\begin{claim}\label{clm:local-forcing}
Let $i\in\{1,2\}$.
The following assertions hold.
\begin{enumerate}[label=\textup{(\roman*)},ref=\textup{(\roman*)}]
\item\label{itm:local-forcing-same}
Let $a,b\in V_i$, and let $c,d$ lie either both in $V_i$ or both in $V_{3-i}$. If $acbd\in\mathcal B$, then 
\[
d_{\mathcal M}(bcd)\ge c_0n 
\quad\text{or}\quad 
\max\{d_{\mathcal M}(ac),\,d_{\mathcal M}(ad),\,d_{\mathcal M}(bd)\}\ge c_0n^2.
\]

\item\label{itm:local-forcing-cross}
Let $a,c\in V_i$ and $b,d\in V_{3-i}$. If $acdb\in\mathcal B$, then 
\[
\max\{d_{\mathcal M}(cdb),\,d_{\mathcal M}(acd)\}\ge c_0n
\quad\text{or}\quad \max\{d_{\mathcal M}(ac),\,d_{\mathcal M}(bd)\}\ge c_0n^2.
\]
\end{enumerate}
\end{claim}

\begin{proof}
For an ordered $m$-tuple $Y=(y_1,\ldots,y_m)$, write $W_j(Y)\coloneqq y_jy_{j+1}y_{j+2}y_{j+3}$, $j\in \mathbb Z/m\mathbb Z$, for the consecutive $4$-tuples in the cyclic order. Let $\sigma_1$, $\sigma_2$, and $\sigma_3$ be the following sequences of length $m$:
\[
\sigma_1\coloneqq000010(0010)^{k-1},\qquad
\sigma_2\coloneqq010111(0111)^{k-1},\qquad
\sigma_3\coloneqq001110(0010)^{k-1}.
\]
In each application, the symbol $0$ denotes $V_i$ and the symbol $1$ denotes $V_{3-i}$, where $i$ is the part specified in the corresponding item of the claim. A consecutive $4$-tuple is \emph{odd} if it contains an odd number of vertices from $V_{3-i}$, and is \emph{even} otherwise. For $\sigma_1$ and $\sigma_2$, the only even $4$-tuples are the first and the last. For $\sigma_3$, the only even $4$-tuples are the first and the fourth.
Equivalently, the relevant consecutive $4$-tuples are as follows; the word in parentheses records the corresponding $0/1$ pattern.
\[
\begin{array}{c|c|c}
\text{sequence} & \text{even }4\text{-tuples} & \text{odd }4\text{-tuples}\\ \hline
\sigma_1 & W_1(Y)\ (0000),\ W_m(Y)\ (0000) & W_2(Y),W_3(Y),\ldots,W_{m-1}(Y)\\
\sigma_2 & W_1(Y)\ (0101),\ W_m(Y)\ (1010) & W_2(Y),W_3(Y),\ldots,W_{m-1}(Y)\\
\sigma_3 & W_1(Y)\ (0011),\ W_4(Y)\ (1100) & W_2(Y),W_3(Y),W_5(Y),\ldots,W_m(Y).
\end{array}
\]

We first prove \cref{itm:local-forcing-same}. Assume first that $c,d\in V_i$. We use $\sigma_1$. Fix $(y_1, y_2, y_3, y_4) \coloneqq (a, c, b, d)$, and let $Y$ range over all ordered $m$-tuples with this initial segment, pairwise distinct vertices, and vertices lying in the parts prescribed by $\sigma_1$. Since both parts have size at least $n/3$, there are at least $(n/4)^{m-4}$ such tuples for all sufficiently large $n$.

Suppose that every odd consecutive $4$-tuple $W_j(Y)$ with $j\ne m$ is an edge of $H$.
Then the present $4$-tuples $\left\{W_j(Y)\colon j\ne m\right\}$ form a copy of $C_{4k+2}^{4-}$.
Here $W_1=acbd$ is present because $acbd\in\mathcal B\subseteq H$, and $W_m$ is the single omitted $4$-tuple.

We next estimate how many admissible tuples are ruled out by missing odd $4$-tuples involving the fixed vertices $a,b,c,d$.
The first such $4$-tuples are $W_2=cbdy_5$, $W_3=bdy_5y_6$, and $W_{m-1}=y_{m-1}y_mac$.
If $W_2$ is missing, then it is a missing template edge containing the fixed triple $bcd$; hence there are at most $d_{\mathcal M}(bcd)$ choices for $y_5$, and then at most $n^{m-5}$ choices for the remaining positions.
Similarly, missing choices for $W_3$ are controlled by the number of missing template edges containing the fixed pair $bd$, and missing choices for $W_{m-1}$ are controlled by the number of missing template edges containing the fixed pair $ac$.
Thus, if
\[
d_{\mathcal M}(bcd)<c_0n,
\qquad d_{\mathcal M}(bd)<c_0n^2,
\qquad\text{and}\qquad d_{\mathcal M}(ac)<c_0n^2
\]
then these three $4$-tuples eliminate fewer than $5c_0n^{m-4}$ admissible tuples.

If $k=1$, then $m=6$, and the remaining odd $4$-tuple involving fixed vertices is $W_4=dy_5y_6a$, which is controlled by $d_{\mathcal M}(ad)$. Hence, if $d_{\mathcal M}(ad)<c_0n^2$, then this $4$-tuple eliminates fewer than $2c_0n^{m-4}$ additional tuples. If $k\ge2$, then the two additional odd $4$-tuples involving fixed vertices are $W_4=dy_5y_6y_7$ and $W_{m-2}=y_{m-2}y_{m-1}y_ma$.
Each contains exactly one fixed vertex. By \cref{clm:vertex-missing}, these two $4$-tuples eliminate at most $12\varepsilon_3 n^{m-4}$ admissible tuples.

It remains to deal with the internal $4$-tuples. When $k=1$, there are no internal $4$-tuples. When $k\ge2$, a fixed missing edge can appear in one prescribed internal $4$-tuple in at most $24n^{m-8}$ admissible tuples. Since $4|\mathcal M|=\sum_{x\in V(H)}d_{\mathcal M}(x)\le \varepsilon_3 n^4$, all internal $4$-tuples together eliminate at most $24m\varepsilon_3 n^{m-4}$ tuples. If the conclusion of \cref{itm:local-forcing-same} fails, then the number of excluded admissible tuples is less than $(10c_0+(12+24m)\varepsilon_3)n^{m-4}$. 
By the hierarchy $0<\varepsilon_3\ll c_0\ll_k1$, this is smaller than $(n/4)^{m-4}$. Hence some admissible tuple remains. For this tuple, all required odd $4$-tuples are present, and therefore it yields a copy of $C_{4k+2}^{4-}$, a contradiction. This proves \cref{itm:local-forcing-same} when $c,d\in V_i$.

Now assume that $c,d\in V_{3-i}$. Use $\sigma_2$ with the same initial segment. The even $4$-tuples are again $W_1$ and $W_m$, and the estimates for the odd $4$-tuples involving fixed vertices and for the internal $4$-tuples are identical after interchanging the two parts. This completes the proof of \cref{itm:local-forcing-same}.

We now prove \cref{itm:local-forcing-cross}. Use $\sigma_3$. Fix $(y_1, y_2, y_3, y_4) \coloneqq (a, c, d, b)$, and let $Y$ range over all admissible ordered $m$-tuples whose vertices lie in the parts prescribed by $\sigma_3$. Again there are at least $(n/4)^{m-4}$ such tuples for all sufficiently large $n$.
For $\sigma_3$, the even $4$-tuples are $W_1$ and $W_4$. Hence, if all odd $4$-tuples of $Y$ are present in $H$, then $\left\{W_j(Y)\colon j\ne4\right\}$ is a copy of $C_{4k+2}^{4-}$, because $W_1=acdb\in\mathcal B\subseteq H$, while $W_4$ is the omitted $4$-tuple.

The odd $4$-tuples involving fixed vertices are $W_2=cdby_5$, $W_3=dby_5y_6$, $W_{m-1}=y_{m-1}y_mac$, and $W_m=y_macd$.
They are controlled respectively by $d_{\mathcal M}(cdb)$, $d_{\mathcal M}(bd)$, $d_{\mathcal M}(ac)$, and $d_{\mathcal M}(acd)$. Thus, if
\[
d_{\mathcal M}(cdb)<c_0n,
\qquad d_{\mathcal M}(acd)<c_0n,
\qquad d_{\mathcal M}(bd)<c_0n^2,
\qquad\text{and}\qquad d_{\mathcal M}(ac)<c_0n^2
\]
then these four $4$-tuples eliminate fewer than $6c_0n^{m-4}$ admissible tuples.

If $k\ge2$, there is one additional odd $4$-tuple involving a fixed vertex, $W_{m-2}=y_{m-2}y_{m-1}y_ma$, which contains exactly one fixed vertex. By \cref{clm:vertex-missing}, it eliminates at most $6\varepsilon_3 n^{m-4}$ admissible tuples. If $k=1$, then $W_{m-2}=W_4$, which is even and omitted. The internal $4$-tuples are controlled exactly as above. Hence, if the conclusion of \cref{itm:local-forcing-cross} fails, the number of excluded admissible tuples is less than $(10c_0+(12+24m)\varepsilon_3)n^{m-4}$, which is smaller than $(n/4)^{m-4}$. Some admissible tuple remains, giving a copy of $C_{4k+2}^{4-}$, a contradiction. This proves \cref{itm:local-forcing-cross}.
\end{proof}

We first use \cref{clm:local-forcing} to bound pair-degrees of bad edges.

\begin{claim}\label{clm:pair-bad}
For every pair $xy\in\binom V2$, we have $d_{\mathcal B}(xy)<\varepsilon_2n^2$. 
\end{claim}

\begin{proof}
Suppose, to the contrary, that $d_{\mathcal B}(ab)\ge\varepsilon_2n^2$ for some pair $ab \in \binom V2$.

First suppose that $a,b$ lie in the same part; by symmetry, take $a,b\in V_1$. Let
\[
L_{11}\coloneqq\left\{cd\in\tbinom{V_1\setminus\{a,b\}}2\colon abcd\in\mathcal B\right\},
\quad\text{and}\quad 
L_{22}\coloneqq\left\{cd\in\tbinom{V_2}2\colon abcd\in\mathcal B\right\}.
\]
One of these graphs, say $L$, has at least $\varepsilon_2n^2/2$ edges. For each $cd\in L$, apply \cref{clm:local-forcing}\labelcref{itm:local-forcing-same} to $acbd$. 
We call a pair $cd$ \emph{blue} if $d_{\mathcal M}(bcd)\ge c_0n$, and \emph{red} otherwise.

If at least half of the edges of $L$ are blue, then each corresponding missing edge contains $b$ and is counted at most three times over the triples $bcd$. Hence
\[
3d_{\mathcal M}(b)\ge \frac{|L|}{2}c_0n\ge \frac{c_0\varepsilon_2}{4}n^3,
\]
contradicting \cref{clm:vertex-missing}. Hence at least half are \emph{red}.
Let $M_{\mathrm{red}}$ be a maximal matching in the \emph{red} subgraph. Since every \emph{red} edge meets a vertex covered by $M_{\mathrm{red}}$, the \emph{red} subgraph has at most $2|M_{\mathrm{red}}|n$ edges. As it has at least $|L|/2\ge\varepsilon_2n^2/4$ edges, we have $|M_{\mathrm{red}}|\ge\varepsilon_2n/8$. For each $cd\in M_{\mathrm{red}}$, choose $p(cd)\in\{ac,ad,bd\}$ with $d_{\mathcal M}(p(cd))\ge c_0n^2$. These selected pairs are distinct. Moreover, by the matching property, a fixed missing edge contributes to $\sum_{cd\in M_{\mathrm{red}}}d_{\mathcal M}(p(cd))$ through at most three selected pairs, and every selected pair contains $a$ or $b$. Hence
\[
3(d_{\mathcal M}(a)+d_{\mathcal M}(b))
\ge \sum_{cd\in M_{\mathrm{red}}}d_{\mathcal M}(p(cd))
\ge |M_{\mathrm{red}}|c_0n^2
\ge \frac{c_0\varepsilon_2}{8}n^3,
\]
again contradicting \cref{clm:vertex-missing}.

It remains to consider the crossing case. Assume $a\in V_1$ and $b\in V_2$. Let $L_{12}$ be the bipartite graph between $V_1\setminus\{a\}$ and $V_2\setminus\{b\}$ whose edges are the pairs $cd$ such that $acdb\in\mathcal B$. Then $|L_{12}|\ge\varepsilon_2n^2$. For each $cd\in L_{12}$, apply \cref{clm:local-forcing}\labelcref{itm:local-forcing-cross}. Put
\[
L_{\mathrm I}\coloneqq\left\{cd\in L_{12}\colon d_{\mathcal M}(cdb)\ge c_0n\right\},
\quad 
L_{\mathrm{II}}\coloneqq\left\{cd\in L_{12}\colon d_{\mathcal M}(acd)\ge c_0n\right\},
\quad\text{and}\quad 
L_{\mathrm R}\coloneqq L_{12}\setminus(L_{\mathrm I}\cup L_{\mathrm{II}}). 
\]
If $|L_{\mathrm I}|\ge |L_{12}|/4$, then each corresponding missing edge contains $b$ and is counted at most three times, so $3d_{\mathcal M}(b)\ge |L_{12}|c_0n/4$, contradicting \cref{clm:vertex-missing}. The same argument with $a$ in place of $b$ excludes $|L_{\mathrm{II}}|\ge |L_{12}|/4$. Hence $|L_{\mathrm R}|>|L_{12}|/2$. Let $M_{\mathrm R}$ be a maximal matching in $L_{\mathrm R}$. Since every edge of $L_{\mathrm R}$ meets a vertex covered by this matching, $|L_{\mathrm R}|\le2|M_{\mathrm R}|n$, and therefore $|M_{\mathrm R}|\ge\varepsilon_2n/4$. For each $cd\in M_{\mathrm R}$, choose $p(cd)\in\{ac,bd\}$ with $d_{\mathcal M}(p(cd))\ge c_0n^2$. These selected pairs are distinct, and each missing edge contributes through at most three selected pairs. Thus
\[
3(d_{\mathcal M}(a)+d_{\mathcal M}(b))
\ge \sum_{cd\in M_{\mathrm R}}d_{\mathcal M}(p(cd))
\ge |M_{\mathrm R}|c_0n^2
\ge \frac{c_0\varepsilon_2}{4}n^3,
\]
contradicting \cref{clm:vertex-missing}. This proves the claim.
\end{proof}

It remains to compare the total contribution of bad edges with the number of missing template edges.  The following \emph{potential function} measures how strongly a bad edge is forced by missing pairs and triples.
For each bad edge $E\in\mathcal B$, define
\[
\Phi(E)\coloneqq\sum_{e\in\binom E2}\frac{d_{\mathcal M}(e)}{d_{\mathcal B}(e)}
+\sum_{f\in\binom E3}\frac{d_{\mathcal M}(f)}{d_{\mathcal B}(f)}.
\]
The denominators are non-zero because every pair and triple appearing in the sums is contained in the bad edge $E$. By double counting,
\[
\begin{aligned}
\sum_{E\in\mathcal B}\Phi(E)
=\sum_{e\in\binom V2:\ d_{\mathcal B}(e)>0} d_{\mathcal M}(e)
 +\sum_{f\in\binom V3:\ d_{\mathcal B}(f)>0} d_{\mathcal M}(f)
\le \sum_{e\in\binom V2}d_{\mathcal M}(e)
    +\sum_{f\in\binom V3}d_{\mathcal M}(f)
=10|\mathcal M|.
\end{aligned}
\]
For each triple $T\in\binom V3$, define $\mathcal B(T)\coloneqq\left\{E\in\mathcal B\colon T\subseteq E\right\}$.

\begin{claim}\label{clm:B-of-T-all}
For every triple $T$ with $\mathcal B(T)\ne\varnothing$, we have $\sum_{E\in\mathcal B(T)}\Phi(E)>10|\mathcal B(T)|$. 
\end{claim}

\begin{proof}
Relabel $T=abc$ so that $a,b$ lie in the same part. Since a bad edge has even intersection with each side of the odd-bipartite template, every edge in $\mathcal B(abc)$ has the form $abcd$, where $c,d$ lie in the same part.

First suppose $|\mathcal B(abc)|<\varepsilon_1n$. For $E=abcd\in\mathcal B(abc)$, apply \cref{clm:local-forcing}\labelcref{itm:local-forcing-same} with $(a',c',b',d')=(d,a,c,b)$, so the bad edge $a'c'b'd'$ is $dacb=abcd$. Then either $d_{\mathcal M}(abc)\ge c_0n$, or some pair $e\subseteq E$ satisfies $d_{\mathcal M}(e)\ge c_0n^2$. In the first case, $\Phi(E)\ge c_0/\varepsilon_1>10$, and in the second case, $\Phi(E)>c_0/\varepsilon_2>10$ by \cref{clm:pair-bad}. Hence the desired inequality holds.

Now suppose $|\mathcal B(abc)|\ge\varepsilon_1n$. Write $N\coloneqq\left\{d\colon abcd\in\mathcal B(abc)\right\}$ and $\mathcal B(abc)=\left\{abcd\colon d\in N\right\}$. 
Call $abcd$ \emph{light} if $d_{\mathcal B}(abd)\le\varepsilon_1n$, and \emph{heavy} otherwise. Let $\mathcal B_{\mathrm L}(abc)$ and $\mathcal B_{\mathrm H}(abc)$ be the corresponding sets.
Each bad edge containing $ab$ contributes to at most two of the numbers $d_{\mathcal B}(abd)$, according to which of its two vertices outside $ab$ is chosen as $d$. Hence
\[
2d_{\mathcal B}(ab)
\ge \sum_{abcd\in\mathcal B_{\mathrm H}(abc)}d_{\mathcal B}(abd)
>\varepsilon_1n|\mathcal B_{\mathrm H}(abc)|.
\]
By \cref{clm:pair-bad} and the hierarchy $2\varepsilon_2\le\varepsilon_1^3$, we have
\[
|\mathcal B_{\mathrm H}(abc)|\le\frac{2d_{\mathcal B}(ab)}{\varepsilon_1n}
\le\varepsilon_1|\mathcal B(abc)|,
\quad\text{and}\quad
|\mathcal B_{\mathrm L}(abc)|\ge(1-\varepsilon_1)|\mathcal B(abc)|.
\]
If $E=abcd$ is \emph{light}, apply \cref{clm:local-forcing}\labelcref{itm:local-forcing-same} with $(a',c',b',d')=(c,a,d,b)$, so the bad edge $a'c'b'd'$ is $cadb=abcd$. Then either $d_{\mathcal M}(abd)\ge c_0n$, giving $\Phi(E)\ge c_0/\varepsilon_1$, or some pair $e\subseteq E$ has $d_{\mathcal M}(e)\ge c_0n^2$, giving $\Phi(E)>c_0/\varepsilon_2\ge c_0/\varepsilon_1$.
Therefore
\[
\sum_{E\in\mathcal B(abc)}\Phi(E)
\ge \frac{c_0}{\varepsilon_1}|\mathcal B_{\mathrm L}(abc)|
\ge \frac{c_0}{\varepsilon_1}(1-\varepsilon_1)|\mathcal B(abc)|
>10|\mathcal B(abc)|,
\]
provided the hierarchy $0<\varepsilon_2\ll\varepsilon_1\ll c_0\ll_k1$ is chosen appropriately. This proves the claim.
\end{proof}

We complete the reduced argument by summing the potential-function estimate over all triples. Assume first that $\mathcal B\ne\varnothing$. Summing \cref{clm:B-of-T-all} over all triples $T$ with $\mathcal B(T)\ne\varnothing$, and using the fact that each bad edge contains exactly four triples, we obtain
\[
4\sum_{E\in\mathcal B}\Phi(E)
=\sum_{T\in\binom V3}\sum_{E\in\mathcal B(T)}\Phi(E)
>10\sum_{T\in\binom V3}|\mathcal B(T)|
=40|\mathcal B|.
\]
Hence $\sum_{E\in\mathcal B}\Phi(E)>10|\mathcal B|$. 
Together with $\sum_{E\in\mathcal B}\Phi(E)\le10|\mathcal M|$, this implies $|\mathcal M|>|\mathcal B|$. By \eqref{eq:H-decomposition-final}, $|H|=|\mathcal T|-|\mathcal M|+|\mathcal B|<|\mathcal T|\le b(n)$, contradicting $|H|\ge b(n)$. Thus $\mathcal B=\varnothing$. Then $|H|=|\mathcal T|-|\mathcal M|$.
Since $|H|\ge b(n)\ge |\mathcal T|$, we have $\mathcal M=\varnothing$ and $|\mathcal T|=b(n)$. Therefore $H=\mathcal T=\mathbb{B}_4^{\mathrm{odd}}(V_1,V_2)$, and $|\mathcal T|=b(n)$.
This proves the reduced assertion required by \cref{lem:min-degree-reduction}. Together with \cref{lem:odd-bip-Fk-free}, this gives both $\ex(n,C_{4k+2}^{4-})=b(n)$ and the equality statement in \cref{THM:main_C_4k+2minus}.
\end{proof}

\section{Concluding remarks}\label{SEC:Remarks}

For full tight cycles not settled here, the complete odd-bipartite construction gives the lower bound $1/2$.
The exact rounded flag algebra computations in \texttt{FourGraph\_C6.ipynb}, \texttt{FourGraph\_C11.ipynb}, and \texttt{FourGraph\_C13.ipynb} give the upper bounds in \cref{TAB:remarks-bounds}.

For the full cycles in \cref{THM:main_C_4k+2}, the complete odd-bipartite construction is asymptotically extremal, but it is not extremal for the exact Tur\'{a}n problem.
The same non-extremality already occurs for $C_6^4$.

\begin{proposition}\label{prop:full-cycle-extra-edges}
Let $k\ge1$ be fixed.  For all sufficiently large $n$, there exists a $C_{4k+2}^{4}$-free $4$-graph on $n$ vertices with at least $b_4(n)+\Omega(n)$ edges.
More precisely, if $A\sqcup B$ is a partition attaining $b_4(n)$, then one can add $\Omega(n)$ further edges to $\mathbb{B}_{4}^{\mathrm{odd}}(A,B)$ without creating a copy of $C_{4k+2}^{4}$.
\end{proposition}

\begin{proof}
Since the edge size is even, interchanging the two parts does not change $\mathbb{B}_{4}^{\mathrm{odd}}(A,B)$, so assume $|A|\ge |B|$.
Fix a pair $P\in\binom A2$, and choose pairwise disjoint pairs
$Q_1,\ldots,Q_t$ in $A\setminus P$, where $t=\lfloor(|A|-2)/2\rfloor$.
Put $\mathcal M\coloneqq\{P\cup Q_i:1\le i\le t\}$ and $H\coloneqq\mathbb{B}_{4}^{\mathrm{odd}}(A,B)\cup\mathcal M$.
Every edge in $\mathcal M$ is a missing edge of the odd-bipartite template, and $|\mathcal M|=t=\Omega(n)$ since $|A|\ge n/2$.
We claim that $H$ is $C_{4k+2}^{4}$-free.

Suppose, for a contradiction, that $v_1,\ldots,v_m$ form a copy of $C_m^4$ in $H$, where $m=4k+2$.
Write $E_i\coloneqq v_iv_{i+1}v_{i+2}v_{i+3}$ for its cyclic edges, where $i\in\mathbb Z/m\mathbb Z$, and let $I\coloneqq\{i:E_i\in\mathcal M\}$.
For each parity class $s\in\{0,1\}$, every vertex of the cycle occurs in exactly two edges $E_i$ with $i\equiv s\pmod2$.
Thus $\sum_{i\equiv s\pmod2}|E_i\cap A|\equiv0\pmod2$.
There are $2k+1$ indices in each parity class.
The edges not in $\mathcal M$ have odd intersection with $A$, while the edges in $\mathcal M$ have even intersection with $A$.
It follows that $|\{i\in I:i\equiv s\pmod2\}|$ is odd for each $s\in\{0,1\}$.
In particular, there are two edges of the copy belonging to $\mathcal M$ whose cyclic indices have opposite parity.

If $m\ge10$, two edges of a tight $m$-cycle with indices of opposite parity cannot intersect in exactly two vertices: their cyclic distance is odd, and their intersection size is $3$, $1$, or $0$.
On the other hand, any two distinct edges of $\mathcal M$ intersect exactly in the fixed pair $P$, a contradiction.

It remains to consider $m=6$.
The set $\{i\in I:i\equiv s\pmod2\}$ cannot have size $3$ for either parity $s$, since the three cycle edges in one parity class have pairwise intersections equal to three different pairs, whereas any two distinct edges of $\mathcal M$ intersect in the same pair $P$.
Thus each parity class contains exactly one edge of $\mathcal M$.
If these two edges have cyclic distance $1$ or $5$, then they intersect in three vertices, again impossible for two distinct edges of $\mathcal M$.
Hence they have cyclic distance $3$.
Their union is the whole vertex set of the $C_6^4$, so all six cycle vertices lie in $A$.
Then every cycle edge has even intersection with $A$, and hence every cycle edge would have to belong to $\mathcal M$; but adjacent cycle edges intersect in three vertices, while distinct edges of $\mathcal M$ intersect in exactly $P$.
This final contradiction proves that $H$ is $C_m^4$-free.
\end{proof}

The cycles minus one edge in the other two congruence classes appear to have a different extremal construction.
Starting from a nearly balanced bipartition $V_1\cup V_2$, one may take all $4$-sets meeting each part in exactly two vertices and then iterate the construction inside both parts, following Sidorenko~\cite{Sid24}.
This gives $C_{\ell}^{4-}$-free examples of density $\frac37$ for all $\ell\equiv1,3\pmod4$.
The flag algebra computations in \texttt{FourGraph\_C7m.ipynb}, \texttt{FourGraph\_C9m\_7vtx.ipynb}, and \texttt{FourGraph\_C11m\_7vtx.ipynb}, after excluding the homomorphic images on at most $6$ vertices used in the scripts, give the corresponding upper bounds in \cref{TAB:remarks-bounds}.

The last three upper bounds are close to $\frac37=0.428571\ldots$ and suggest that the iterated bipartite construction may be the correct extremal example for $C_{\ell}^{4-}$ when $\ell\equiv1,3\pmod4$.
Determining these densities remains open.

\begin{table}[H]
\centering
\begingroup
\renewcommand{\arraystretch}{1.2}
\begin{tabular}{ccc}
\toprule
Forbidden configuration & construction lower bound & flag algebra upper bound \\
\midrule
$C_6^4$ & $1/2$ & $673753/1152000=0.584855\ldots$ \\
$C_{11}^4$ & $1/2$ & $363463/691200=0.525843\ldots$ \\
$C_{13}^4$ & $1/2$ & $13879/27648=0.501989\ldots$ \\
$C_{7}^{4-}$ & $3/7$ & $303349/691200=0.438872\ldots$ \\
$C_{9}^{4-}$ & $3/7$ & $693253/1612800=0.429844\ldots$ \\
$C_{11}^{4-}$ & $3/7$ & $49519/115200=0.429852\ldots$ \\
\bottomrule
\end{tabular}
\endgroup
\caption{Bounds obtained from the constructions described above and from the exact rounded flag algebra computations.}
\label{TAB:remarks-bounds}
\end{table}

%%%%%%%%%%%%%%%%%%%%%%%%%%%%%%%%%%%%%%%%%%%%%%
\section*{Acknowledgments}
\begin{sloppypar}
X.L. would like to thank Levente Bodn\'ar for helping us verify the slackness-related results at an early stage of this project. 
J.H. was supported by the National Key R\&D Program of China (No.~2023YFA1010202). 
X.L. was supported by the Excellent Young Talents Program (Overseas) of the National Natural Science Foundation of China. 
Y.Z. was supported by the Doctoral Student Program of the Young S\&T Talents Cultivation Project, CAST. 
\end{sloppypar}
%%%%%%%%%%%%%%%%%%%%%%%%%%%%%%%%%%%%%%%%%%%%%%
\bibliographystyle{abbrv}
\bibliography{TightCycle}
%%%%%%%%%%%%%%%%%%%%%%%%%%%%%%%%%%%%%%%%%%%%%
%%%%%%%%%%%%%%%%%%%%%%%%%%%%%%%%%%%%%%%%%%%%%%

\appendix
\titleformat{\section}[display]{\normalfont\large\bfseries}{\Large Appendix}{0.35em}{}

\section{A direct upper-bound proof for \texorpdfstring{$C_6^{4-}$}{C6 minus}}\label{APP:C6minus-direct-upper}

This appendix gives a self-contained proof of $\pi(C_6^{4-})\le1/2$ by writing out the first positive semidefinite block of the $C_6^{4-}$ certificate.
Here we include a human-readable proof produced by AI by giving it the certificate generated by the flag algebra script.
The proof converts each induced six-vertex $4$-graph into an ordinary graph by taking complements of pairs, proves a small ordinary graph inequality, and then averages the resulting six-vertex inequality.

Recall that $C_6^{4-}$ has vertex set $\{1,\ldots,6\}$ and edge set $\{1234,2345,3456,1456,1256\}$.

\begin{lemma}\label{LEM:appendix-path-chvatal}
Let $G$ be an ordinary graph on six vertices with degree sequence $d_1\le d_2\le \cdots\le d_6$.
If, for each $k=1,2,3$, either $d_k\ge k$ or $d_{7-k}\ge6-k$, then $G$ has a Hamilton path.
\end{lemma}

\begin{proof}
We use the standard Chvatal-type Hamilton-cycle criterion in the following form: if a graph $J$ on $N$ vertices has degree sequence $a_1\le\cdots\le a_N$ and, for every $i<N/2$, either $a_i\ge i+1$ or $a_{N-i}\ge N-i$, then $J$ is Hamiltonian.
For completeness, we recall the short proof.

Suppose, for a contradiction, that $J$ satisfies the displayed degree condition but is not Hamiltonian.
Add edges until we obtain an edge-maximal non-Hamiltonian graph $J'$, and write its degree sequence as $a'_1\le\cdots\le a'_N$.
The degree condition still holds for $J'$, since all degrees have only increased.
Choose a non-edge $uv$ of $J'$ for which $d(u)+d(v)$ is as large as possible, and write $r=d(u)\le d(v)=s$.
The graph $J'+uv$ has a Hamilton cycle using $uv$, and deleting this new edge gives a Hamilton path $u=x_1,x_2,\ldots,x_N=v$ in $J'$.
If both $ux_{i+1}$ and $vx_i$ were edges of $J'$ for some $i$, then these two chords would close a Hamilton cycle in $J'$, a contradiction.
Thus the two sets $\{i: ux_{i+1}\in E(J')\}$ and $\{i: vx_i\in E(J')\}$ are disjoint subsets of $\{1,\ldots,N-1\}$, so $r+s\le N-1$.
In particular $r<N/2$, and if $r=0$, the criterion already fails for $i=1$.

Assume therefore that $r\ge1$.
By the choice of $uv$, every non-neighbour of $v$ has degree at most $r$; since $v$ has $N-1-s\ge r$ non-neighbours, we get $a'_r\le r$.
Similarly every non-neighbour of $u$ has degree at most $s$, and together with $u$ this gives $N-r$ vertices of degree at most $s$.
Hence $a'_{N-r}\le s\le N-1-r<N-r$, so the criterion fails for $i=r$, a contradiction.

Now add one universal vertex to $G$.
The resulting graph has seven vertices and degree sequence $d_1+1\le\cdots\le d_6+1\le6$.
For $N=7$, the Hamilton-cycle criterion is checked only for $i=1,2,3$, where it is exactly the hypothesis of the lemma.
Thus the augmented graph has a Hamilton cycle, and deleting the universal vertex from this cycle gives a Hamilton path in $G$.
\end{proof}

\begin{lemma}\label{LEM:appendix-P6-disjoint-pairs}
Let $G$ be an ordinary graph on six vertices with no copy of $P_6$.
If $m=e(G)$ and $M$ is the number of unordered pairs of vertex-disjoint edges of $G$, then $2M\le3m$.
\end{lemma}

\begin{proof}
Let the degrees of $G$ be $d_1,\ldots,d_6$.
For an edge $uv\in E(G)$, exactly $m-d(u)-d(v)+1$ edges are disjoint from $uv$.
Summing over all edges gives
\begin{equation}\label{EQ:appendix-disjoint-pair-count}
2M=m(m+1)-\sum_{i=1}^6 d_i^2.
\end{equation}
It is therefore enough to prove
\begin{equation}\label{EQ:appendix-degree-square}
\sum_{i=1}^6 d_i^2\ge m(m-2).
\end{equation}

Write the degree sequence in nondecreasing order.
Since $G$ has no Hamilton path, \cref{LEM:appendix-path-chvatal} implies that at least one of the following holds: \textup{(A)} $d_1=0$ and $d_6\le4$; \textup{(B)} $d_2\le1$ and $d_5\le3$; or \textup{(C)} $d_3\le2$ and $d_4\le2$.

In Case \textup{(A)}, one vertex is isolated and all other degrees are at most $4$, so $m\le\binom52=10$.
Cauchy's inequality on the five non-isolated positions gives $\sum_i d_i^2\ge(2m)^2/5\ge m(m-2)$.

In Case \textup{(B)}, we have $2m=\sum_i d_i\le1+1+3+3+3+5=16$, so $m\le8$.
If $m\le6$, then Cauchy's inequality gives $\sum_i d_i^2\ge(2m)^2/6\ge m(m-2)$.
If $m=7$, putting $s=d_1+d_2$ gives $0\le s\le2$ and
\[
\sum_i d_i^2\ge \frac{s^2}{2}+\frac{(14-s)^2}{4}\ge38>35=m(m-2).
\]
If $m=8$, the degree-sum bound is tight, so the degree sequence is $(1,1,3,3,3,5)$ and $\sum_i d_i^2=54>48=m(m-2)$.

It remains to consider Case \textup{(C)}.
Here $d_1,d_2,d_3,d_4\le2$, so $m\le9$.
The case $m\le6$ is again immediate from Cauchy's inequality.
For $m=7,8,9$, let $s=d_1+d_2+d_3+d_4$.
Since $0\le s\le8$, Cauchy's inequality gives $\sum_i d_i^2\ge s^2/4+(2m-s)^2/2$.
This is at least $48=m(m-2)$ for $m=8$ and at least $66>63=m(m-2)$ for $m=9$.
For $m=7$, it is at least $34$.
The value $34$ can occur only for degree sequence $(2,2,2,2,3,3)$; otherwise $\sum_i d_i^2\equiv \sum_i d_i\pmod2$ gives $\sum_i d_i^2\ge36>35=m(m-2)$.

It remains only to rule out this exceptional sequence.
Such a graph is connected, since otherwise the minimum degree $2$ would force two triangular components and all degrees would be $2$.
Also, $G$ contains a path on at least five vertices.
Indeed, let $w_1\cdots w_t$ be a longest path in a connected graph of minimum degree $\delta$.
All neighbours of $w_1$ and $w_t$ lie on this path.
If the path is not spanning, then the sets $N(w_1)^+=\{w_{i+1}:w_i\in N(w_1)\}$ and $N(w_t)$ must be disjoint: an intersection closes a cycle on the path vertices, and connectedness then gives either a spanning cycle or a longer path.
Hence $t-1\ge2\delta$.
Here $\delta(G)\ge2$, and therefore $G$ has a path on at least five vertices.
Let $v_1v_2v_3v_4v_5$ be a longest path in $G$, and let $x$ be the sixth vertex.
The vertex $x$ is adjacent neither to $v_1$ nor to $v_5$, and it cannot be adjacent to two consecutive vertices of the path; otherwise there would be a Hamilton path.
Since $d(x)\ge2$, we must have $N(x)=\{v_2,v_4\}$.
To avoid a Hamilton path, the chords $v_1v_3$, $v_1v_5$, and $v_3v_5$ are forbidden.
The endpoints $v_1$ and $v_5$ therefore need the chords $v_1v_4$ and $v_2v_5$, respectively, in order to have degree at least $2$.
But then $v_2$ and $v_4$ both have degree at least $4$, contradicting the degree sequence $(2,2,2,2,3,3)$.

Thus \eqref{EQ:appendix-degree-square} holds in all cases, and \eqref{EQ:appendix-disjoint-pair-count} gives $2M\le3m$.
\end{proof}

\begin{proposition}\label{PROP:appendix-C6minus-upper}
We have $\pi(C_6^{4-})\le1/2$.
\end{proposition}

\begin{proof}
Let $H$ be a $C_6^{4-}$-free $4$-graph on vertex set $V$, where $|V|=n$.
For each six-set $S\in\binom V6$, define the complement-pair graph $G_S$ on vertex set $S$ by declaring $xy\in E(G_S)$ if and only if $S\setminus\{x,y\}\in E(H)$.
Let $m(S)=e(G_S)=e(H[S])$, and let $M(S)$ be the number of unordered pairs of vertex-disjoint edges of $G_S$.

The complements of the five edges of $C_6^{4-}$ are $56,16,12,23,34$, which form the ordinary path $5-6-1-2-3-4$.
Thus a copy of $P_6$ in $G_S$ would give a copy of $C_6^{4-}$ in $H[S]$.
Since $H$ is $C_6^{4-}$-free, each $G_S$ is $P_6$-free, and \cref{LEM:appendix-P6-disjoint-pairs} gives $2M(S)\le3m(S)$ for every six-set $S$.

For a pair $p\in\binom S2$, set $\xi_S(p)=-1$ if $p\in E(G_S)$ and $\xi_S(p)=1$ otherwise.
Define
\[
Q(S)=\frac1{90}
\sum_{\substack{\{p,q\}\subseteq\binom S2\\ p\cap q=\varnothing}}
\xi_S(p)\xi_S(q),
\]
where the sum is over unordered pairs of disjoint pairs in $S$.
There are $45$ such pair-pairs.
If $m=m(S)$ and $M=M(S)$, then $6m-2M$ of them contain exactly one edge of $G_S$, and hence $Q(S)=(45-12m(S)+4M(S))/90$.
Therefore
\[
\frac12-\frac{m(S)}{15}-Q(S)
=\frac{3m(S)-2M(S)}{45}\ge0.
\]
Equivalently, every six-set satisfies $m(S)/15\le1/2-Q(S)$.

Averaging this local inequality over all six-sets gives
\begin{equation}\label{EQ:appendix-density-average}
\frac{e(H)}{\binom n4}
\le
\frac12-\mathbb E_{S\in\binom V6} Q(S),
\end{equation}
because every edge of $H$ is contained in exactly $\binom{n-4}{2}$ six-sets.
It remains to show that $\mathbb E_SQ(S)\ge -O(1/n)$.

Fix a pair $R\in\binom V2$ and, for $P\in\binom{V\setminus R}{2}$, put $\eta_R(P)=1-2\mathbf 1_{R\cup P\in E(H)}$.
Then
\[
\mathbb E_{S\in\binom V6}Q(S)
=\frac{1}{90\binom n6}
\sum_{R\in\binom V2}
\sum_{\substack{\{P,Q\}\subseteq\binom{V\setminus R}{2}\\ P\cap Q=\varnothing}}
\eta_R(P)\eta_R(Q).
\]
For fixed $R$, write $X_R=\binom{V\setminus R}{2}$.
If the disjointness condition is ignored, then
\[
\sum_{\substack{\{P,Q\}\subseteq X_R\\ P\ne Q}}
\eta_R(P)\eta_R(Q)
=\frac12\left(\left(\sum_{P\in X_R}\eta_R(P)\right)^2-|X_R|\right)
\ge -\frac{|X_R|}{2}.
\]
There are only $O(n^3)$ unordered pairs $\{P,Q\}\subseteq X_R$ with $P\cap Q\ne\varnothing$, and each summand has absolute value $1$.
Thus the inner sum with the disjointness condition is at least $-O(n^3)$, uniformly in $R$.
Since there are $O(n^2)$ choices of $R$ and $\binom n6=\Theta(n^6)$, we get $\mathbb E_SQ(S)\ge -O(1/n)$.
Together with \eqref{EQ:appendix-density-average}, this gives $e(H)/\binom n4\le1/2+O(1/n)$.
Taking the limit over $n$ proves the proposition.
\end{proof}
\end{document}